\documentclass[12pt, a4j]{amsart}
\title{Minimizing CM degree and specially K-stable varieties}
\author{Masafumi Hattori}
\address{Department of Mathematics, Faculty of Science, Kyoto University, Kyoto, 606-8502, Japan}
\email{hattori.masafumi.47z@st.kyoto-u.ac.jp}

\usepackage{amsthm}
\usepackage{amsmath}
\usepackage{amssymb}
\usepackage{amsfonts}
\usepackage{mathrsfs}
\usepackage{enumerate}
\usepackage{amscd}
\usepackage{color}
\usepackage{graphicx}
\usepackage[all]{xy}
\usepackage{hyperref}
\usepackage{fancyhdr}
\usepackage[top=30truemm,bottom=30truemm,left=30truemm,right=30truemm]{geometry}

\hypersetup{colorlinks=true}

\def\coloneq{\mathrel{\mathop:}=}
\newtheorem{thm}{Theorem}[section]
\newtheorem{lem}[thm]{Lemma}
\newtheorem{prop}[thm]{Proposition}
\newtheorem{conj}[thm]{Conjecture}
\newtheorem{claim}{Claim}

\newtheorem{cor}[thm]{Corollary}

\theoremstyle{definition}
\newtheorem{de}[thm]{Definition}
\newtheorem{ass}[thm]{Condition}
\newtheorem{ex}[thm]{Example}

\newtheorem{rem}[thm]{Remark}

\begin{document}
\maketitle
\begin{abstract}
We prove that the 
degree of the CM line bundle for a normal family over a curve with fixed general fibers is 
strictly minimized if 
the special fiber is 
either 
\begin{itemize}
    \item a smooth projective manifold with a unique  cscK metric or 
    \item ``specially K-stable", which is a new class we introduce in this paper. 
\end{itemize}
This phenomenon, 
as conjectured by Odaka (cf., \cite{O5}), 
is a quantitative  strengthening of the 
separatedness conjecture 
of moduli spaces of polarized K-stable varieties. 

The above mentioned 
special K-stability implies the original K-stability and a lot of cases satisfy it e.g., K-stable log Fano, klt Calabi-Yau (i.e., $K_X\equiv0$), lc varieties with the ample canonical divisor and uniformly adiabatically K-stable klt-trivial fibrations over curves (cf., \cite{Hat}). 
\end{abstract}
\section{Introduction}
We work over $\mathbb{C}$ but all results in \S \ref{3} except Corollary \ref{cmcsck} or Theorem \ref{spec} also hold for any algebraically closed field with characteristic 0. 

\subsection{Separatedness of moduli spaces of K-stable varieties}\label{intro}
To construct moduli spaces of polarized algebraic varieties, the following condition is one of the most important ingredients and guarantees ``separatedness" in the following sense.
\begin{itemize}
\item[$\clubsuit$] Let $(X,L)\to C$ and $(X',L')\to C$ be two proper flat normal families of $n$-dimensional polarized varieties of a certain class $(*)$ over a smooth curve. If the generic fibers of the two families coincide, then their special fibers $(X_0,L_0)$ and $(X'_0,L'_0)$ are ``equivalent" in some sense over a closed point $0\in C$. 
\end{itemize}
This is proved in a few cases, for example, when $(*)$ is the class of stable curves in \cite{DM}. For general K-ample ($K_X$ is ample literally) slc pairs, $\clubsuit$ holds similarly by the theory of MMP (cf., \cite{KSB}, \cite{Ko2}). They are known to be K-stable by \cite{O}. It is also proved in the unpublished note of Boucksom \cite{Bou2} that $\clubsuit$ holds for klt minimal models in a similar way. On the other hand, it is easy to see that $\clubsuit$ does not hold in general at least for K-unstable Fano varieties by \cite{LX}. In the recent studies of Fano varieties, K-stability plays an important role in construction of the moduli space of K-(poly)stable Fano varieties (so-called K-moduli cf., \cite{X}), and it is proved that $\clubsuit$ holds when ``equivalence'' is S-equivalence for K-semistable Fano varieties by Blum and Xu \cite{BX}. However, to check whether $\clubsuit$ holds or not for any class $(*)$ has still been one of the most challenging problems in algebraic geometry.

\subsection{K-stability and CM minimization}\label{1.2*}K-stability was originally introduced by \cite{T}, \cite{Dn2} in K\"{a}hler geometry to study when constant scalar curvature K\"{a}hler (for short cscK) metrics exist. 
Note that K-stability can be rephrasable as follows (cf., Definition \ref{defkst23}). If the trivial test configuration minimizes the Donaldson-Futaki (DF) invariants of normal test configurations in the strict sense, then $(X,L)$ is K-stable. Roughly speaking, this is one of the algebro-geometric counterparts of the result of \cite{Mab} on Fano manifolds, which states that the K-energy takes a local minimum at a K\"{a}hler-Einstein metric. We remark that it is known by 
\cite{CC}, if the K-energy is proper and $\mathrm{Aut}(X,L)$ is discrete, it attains a unique global minimum at a unique cscK metric. So to speak, K-stability is characterized by ``DF minimization" in the sense of Conjecture \ref{mainconj} below. 

On the other hand, Paul and Tian \cite{PT} introduced the Chow-Mumford (CM) line bundle, which is a $\mathbb{Q}$-line bundle defined on the base of a flat family of polarized varieties. Note that the degree of the CM line bundle over a curve, which we call the CM degree, is a generalization of the DF invariants of test configurations (cf., \cite[Lemma 2.5]{FR}).

Odaka proposed the following on CM degrees, which he called the {\it CM minimization conjecture}. 

 \begin{conj}[CM minimization, cf., {\cite[Conjecture 8.1]{O5}}]\label{mainconj}
 Let $\pi:(X,L)\to C$ be a polarized family over a smooth projective curve $C$ such that $(X_0,L_0)$ is K-semistable  (cf., Definition \ref{notate}). Let $\mathrm{CM}((X,L)/C)$ be the CM degree. Then 
 \[
 \mathrm{CM}((X,L)/C)\le \mathrm{CM}((X',L')/C)
 \]
 for any polarized family $\pi':(X',L')\to C$ such that there exists a $C^\circ$-isomorphism $f^{\circ}:(X,L)\times_CC^{\circ}\cong (X',L')\times_CC^{\circ}$. 
 
 Furthermore, if $(X_0,L_0)$ is K-stable and $X'$ is normal, then equality holds iff $f^{\circ}$ can be extended to $f:(X,L)\cong(X',L')$ over $C$ entirely.
 \end{conj}
 
Taking what we explained in the first paragraph of \S\ref{1.2*} into account, Conjecture \ref{mainconj} predicts that K-stability would be characterized not only by DF minimization but also by CM minimization. Conjecture \ref{mainconj} also predicts that if we chose $(*)$ to be the class of K-stable varieties in $\clubsuit$, then we would obtain separatedness automatically. This conjecture was indeed proved for lc K-ample and klt Calabi-Yau ($K_X\equiv0$) varieties by Wang and Xu \cite{WX} and by Odaka \cite{O2} respectively. Furthermore, for K-semistable Fano varieties, the above conjecture holds as shown by Xu \cite{X}. Thus, the results on separatedness in \S\ref{intro} except \cite{Bou2} follow from Conjecture \ref{mainconj} in special cases. In \cite{WX} and \cite{O2}, Conjecture \ref{mainconj} is proved by the Hodge index theorem and by the observation of the log discrepancy. On the other hand, the proof of Conjecture \ref{mainconj} for K-stable Fano varieties relies heavily on the result of \cite{LX}. Unfortunately, their methods can not be applied to families of more general polarized varieties directly. On the other hand, Ohno \cite{Oh} studied the opposite direction of Conjecture \ref{mainconj}. That is, he proved that if the CM degree takes a minimum then the special fiber is necessarily slope K-semistable under a certain condition. 
 
 The aim of this paper is to confirm that Conjecture \ref{mainconj} holds for many cases. Our first result is to settle Conjecture \ref{mainconj} for the following case that seems to be quite meaningful to K\"{a}hler geometry.
 \begin{thm}[$=$ {Corollary \ref{cmcsck}}]\label{mt1}
  Conjecture \ref{mainconj} holds if $X_0$ is smooth and $(X_0,L_0)$ has the discrete automorphism group and a cscK metric.
 \end{thm}
 
 According to this theorem, Conjecture \ref{mainconj} seems to be quite natural. Indeed, if $(X_0,L_0)$ satisfies the assumption of Theorem \ref{mt1}, it is known that $(X_0,L_0)$ is K-stable by \cite{St} and the Yau-Tian-Donaldson conjecture predicts that the converse would hold. 

 On the other hand, we introduce a new class of K-stable varieties, {\it specially }K-{\it stable varieties} (cf., Definition \ref{mmt}). K-stable log Fano, slc K-ample and klt Calabi-Yau varieties are contained in this class. Furthermore, some of polarized varieties confirmed to be uniformly K-stable in the previous works by the author (cf., \cite{Hat2}, \cite{Hat}) are also specially K-stable, for example, klt minimal models and uniformly adiabatically K-stable klt-trivial fibrations over curves (see Theorem \ref{spec}).
 
  Our second result confirms that Conjecture \ref{mainconj} also holds for specially K-stable varieties as follows.
 \begin{thm}[$=$ {Theorem \ref{CM2}}]\label{mt2}
  In Conjecture \ref{mainconj}, let $K_X$ be $\mathbb{Q}$-Cartier. If $(X_0,L_0)$ is specially K-semistable, then the following inequality holds
  \[
 \mathrm{CM}((X,L)/C)\le \mathrm{CM}((X',L')/C).
 \]
If $(X_0,L_0)$ is further specially K-stable and $X'$ is normal, then equality holds iff $f^{\circ}$ extends to $f:(X,L)\cong(X',L')$ over $C$ entirely.
 \end{thm} 
 
 We obtain that $\clubsuit$ holds for special K-stable varieties as an immediate corollary ($=$ Corollary \ref{sep}).
 
We can define special K-stability in an intrinsic way by using the $\delta$-invariant (cf., \cite{FO}, \cite{BlJ}) and J-positivity (cf., \cite{G}, \cite{DP}, \cite[Definition 1.1]{S}, \cite{Hat2} and Definition \ref{mmt}) rather than by using the DF invariants of test configurations. Thus, to check special K-stability is much easier than the original K-stability. Furthermore, thanks to Corollary \ref{sep}, we could construct moduli spaces of certain classes of specially K-stable polarized varieties as Deligne-Mumford stacks if we knew openness and boundedness. In fact, the moduli spaces constructed by Hashizume and the author in \cite{HH} parametrize uniformly adiabatically K-stable klt-trivial fibrations over curves and Theorems \ref{mt2} and \ref{spec} guarantee separatedness of these moduli spaces in a different way. Furthermore, we conclude that the result of \cite{Bou2} follows from Theorem \ref{mt2}.
 
\subsection{The technical heart of the proof of the main theorems}\label{heart} Let $\pi:(X,L)\to C$ and $\pi':(X',L')\to C$ be two families generically isomorphic over $C$. As in Conjecture \ref{mainconj}, let $0\in C$ be a special point. We consider when $X_0$ is normal and irreducible and restrict $C$ to an open neighborhood of $0$. Then we define the following filtration as
 \[
 \mathscr{F}^{-i}H^0(X_0,mL_0)=\mathrm{Im}\,(H^0(X',mL'+i(X'_0-\hat{X}_0))\to H^0(X_0,mL_0))
 \]
 for $i\ge0$. Otherwise, we define $ \mathscr{F}^{-i}H^0(X_0,mL_0)=0$. Here, $\hat{X}_0$ is the strict transformation of $X_0$. This filtration is firstly studied in \cite[\S5]{BX} when $(X_0,L_0)$ and $(X'_0,L'_0)$ are K-semistable Fano varieties, and Blum and Xu proved that $\mathscr{F}$ is finitely generated in this case. However, this filtration has not been fully considered yet in general cases. In this paper, we construct such filtrations in general settings. In contrast to the case treated in \cite[\S 5]{BX}, $ \mathscr{F}$ might not be finitely generated. However, we show that $\mathscr{F}$ has the weight function $w_{\mathscr{F}}(m)=b_0m^{n+1}+b_1m^{n}+O(m^{n-1})$ where $n=\textrm{dim}\,X$. Thus, we define the DF invariant $\mathrm{DF}(\mathscr{F})$ of $\mathscr{F}$ in the same way we defined those of test configurations. To prove our conjecture, we observe that the difference of the CM degrees coincides with $\mathrm{DF}(\mathscr{F})$. Thus, we want to compute $\mathrm{DF}(\mathscr{F})$ by approximating via finitely generated ones. However, there is a subtlety that $b_0$ is preserved when we take the limit but $b_1$ is not known to be so. This problem is called the conjecture of regularization of non-Archimedean entropy (cf., \cite[Conjecture 2.5]{BJ2}, \cite[Conjecture 1.8]{Li}). Fortunately, if $(X_0,L_0)$ is smooth and has a cscK metric and the discrete automorphism group, then we take a lower bound as Chow$_\infty(\mathscr{F})$ of $\mathrm{DF}(\mathscr{F})$ by \cite[Proposition 11]{Sz}. See the details in Theorem \ref{CM} and Corollary \ref{cmcsck}. On the other hand, for specially K-stable varieties, we know by \cite[Appendix]{Hat} that we can give a lower bound of the DF invariant of a test configuration as the sum of the log-twisted Ding invariant (cf., \cite{Ber}) and the non-Archimedean (for short., nA) J-functional introduced by \cite{LS}. For Ding invariants, as studied in \cite{Fjt2} and \cite{Fjt}, filtrations play important roles. On the other hand, we see that nA J-functionals are compatible with taking the limit of finitely generated filtrations (cf., Lemma \ref{lim}). Then, we decompose CM degrees into log-twisted Ding degrees and J-degrees, which are generalizations of nA J-functionals, and we obtain a lower bound of the difference of two CM degrees.
 \subsection{Structure of this paper} In \S\ref{2}, we introduce {\it good filtrations}. A good filtration is defined to be a filtration with the weight function that is close to a polynomial with an error term $O(m^{n-1})$. 
  We define the DF invariants of these filtrations in a different way from \cite{Sz} (cf., Definition \ref{goodfil}, Remark \ref{szerem}). On the other hand, we have to consider the volumes of linear series on reducible or non-reduced schemes. There is a powerful tool, the Okounkov body (cf., \cite{LM}, \cite{BC}), to discuss the volumes of linear series of varieties. However, the theory of Okounkov bodies might not work well for reducible or non-reduced schemes. For this, we work on the weight functions of filtered linear series of general schemes.
 
 In \S\ref{3}, we first establish the formula as we stated in \S\ref{heart}. We also establish the log version of this formula in Corollary \ref{logCm}. Here, note that $\Delta_0$, the fiber of the boundary over $0$, might not be integral in general. Then, we apply the theory on the weight functions of filtrations of reducible or non-reduced polarized schemes constructed in \S\ref{nonreduced} to obtain our formulae. Theorem \ref{mt1} follows from Theorem \ref{CM} and from the result on Chow$_\infty$ of \cite{Sz}.
 
 The proof of Theorem \ref{mt2} is more complicated than that of Theorem \ref{mt1}. In \S\ref{mtp2}, we prove Theorem \ref{mt2} in the three steps. We first decompose the CM degree into the log-twisted Ding degree and the J-degree of a family.
 
  In \S\ref{1st}, we consider ``J-minimization". As studied in \cite{Hat2}, nA J-functionals are not affected by singularities. This is a difference between J-stability and K-stability in the sense of \cite{Gitod}. For this, no problem like regularization of nA entropy occurs when we consider nA J-functionals. Thus we define the nA J-functional of a non finitely generated filtration by taking the limit of a sequence of those of finitely generated filtrations (cf., Definition \ref{compatible.divisor}, Lemma \ref{lim}). With this in mind, we prove J-minimization (Proposition \ref{jcm2}) by applying Corollary \ref{logCm}.  
  
  Next, we consider ``Ding minimization" in \S\ref{2nd}. For this, we construct the following new method to prove the implication $\delta(X,-K_X)\ge1\Rightarrow$ Ding-semistability of Fano varieties more directly than \cite{FO} (cf., \cite[Theorem 5.1]{Fjt}) without applying the result of \cite{LX}. The reason why we need the new method is that there is a subtle problem that we can not make use of MMP directly for general log twisted Fano pairs as \cite{LX} or \cite{BLZ} since the twist term can be anti-ample. Let us explain the method briefly for test configurtaions. Let $(\mathcal{X},\mathcal{L})$ be a semiample test configuration for a Fano manifold $(X,-K_X)$ and $\mathfrak{a}=\sum t^i\mathfrak{a}_i$ be an ideal (which is called a flag ideal in \cite{O3}) such that $(\mathcal{X},\mathcal{L})$ is the blow up of $\mathfrak{a}$. Here $t$ is the canonical coordinate of $\mathbb{A}^1$. As in the proof of \cite[Theorem 4.1]{Fjt}, we have
  \[
  \mathrm{Ding}(\mathcal{X},\mathcal{L})=\mathrm{lct}(X\times\mathbb{A}^1,\mathfrak{a};X\times\{0\})-1-\frac{\mathcal{L}^{n+1}}{(n+1)L^n}.
  \]
  Then we relate $\frac{\mathcal{L}^{n+1}}{(n+1)L^n}$ to the asymptotic behavior of the $\delta_k$-invariant for sufficiently large $k$. Thus, we deduce Ding-semistability from $\delta(X,-K_X)\ge1$. To show Ding minimization, we generalize our method to any family over a curve in the log-twisted setting.
  
  Finally, we combine the results on J-minimization and Ding minimization to obtain Theorem \ref{mt2} and explain applications in \S\ref{3.3.3}.
  
\section*{Acknowledgement}
The author would like to thank his advisor Professor Yuji Odaka for his constant support and warm encouragement. The author also would like to thank Odaka and Doctor Eiji Inoue for careful reading the draft and fruitful discussions. This work was supported by JSPS KAKENHI Grant Number JP22J20059.
\section{Preliminaries}\label{2}

We assume that a polarized scheme $(X,L)$ is proper over $\mathbb{C}$, connected and equidimensional of $n$. If we call $X$ is a variety, we further assume that $X$ is irreducible and reduced. Unless otherwise stated, we understand $L$ to be a $\mathbb{Q}$-line bundle, i.e.,  $rL$ is an ample line bundle for some $r\in\mathbb{Z}_{\ge0}$. We denote the intersection product as $L^m\cdot H^{n-m}$ and we understand $mL=L^{\otimes m}$.

\subsection{K-stability and test configurations}\label{defkst}
First, we recall some basic concepts. 

\begin{de}
Let $(X,L)$ be a polarized reduced scheme that satisfies the Serre's condition $S_2$. Suppose that $X$ is smooth or normal crossing in codimension 1 points. Let $\Delta$ be an effective $\mathbb{Q}$-Weil divisor on $X$ such that no irreducible component of $\mathrm{Supp}\, \Delta$ is contained in the singular locus of $X$ and $K_X+\Delta$ is $\mathbb{Q}$-Cartier. Then we call $(X,\Delta,L)$ a {\it polarized deminomal pair}.

If $X$ is further a normal variety, then we call $(X,\Delta,L)$ a polarized normal pair.
\end{de}
We recall log discrepancies, $\delta$-invariants and singularities of pairs as follows.
\begin{de}\label{desing}
First, let $(X,\Delta,L)$ be a polarized normal pair. For any prime divisor $F$ over $X$, we define the {\it log discrepancy} $A_{(X,\Delta)}(F)$ with respect to $F$ as follows. Choose a projective birational morphism $\pi:Y\to X$ from a normal variety $Y$ on which $F$ is a prime divisor defined. Then
\[
A_{(X,\Delta)}(F)=1+\mathrm{ord}_F(K_Y-\pi^*(K_X+\Delta)).
\]
This is independent from the choice of $\pi$. Then we say $(X,\Delta)$ is 
\begin{itemize}
\item {\it klt} if $A_{(X,\Delta)}(F)>0$ for any $F$,
\item {\it lc} if $A_{(X,\Delta)}(F)\ge0$ for any $F$.
\end{itemize}
We remark that if $\Delta$ is noneffective, we define the log discrepancy in the same way and we say $(X,\Delta)$ is sublc if $A_{(X,\Delta)}(F)\ge0$ for any $F$.

For any effective Cartier divisor $D$, we define the {\it log canonical threshold} of $(X,\Delta)$ with respect to $D$
\[
\mathrm{lct}(X,\Delta;D)=\sup\{t\in\mathbb{Q}|(X,\Delta+tD)\textrm{ is a sublc pair}\}.
\]

Next, we define the $\delta$-invariant of a polarized lc pair $(X,\Delta,L)$ as follows. If $(X,\Delta)$ is not klt, then set $\delta(X,\Delta,L)=0$. Otherwise, take $r_0\in\mathbb{Z}_{>0}$ such that $r_0L$ is an ample Cartier divisor. For $m\in\mathbb{Z}_{>0}$, we call $D$ an $r_0m$-basis type divisor if $D=\frac{1}{r_0mh^0(X,r_0mL)}\sum_{i=1}^{h^0(X,r_0mL)}D_i$ where $\{D_i\}_{i=1}^{h^0(X,r_0mL)}$ forms a basis of $H^0(X,r_0mL)$. Then, set
\[
\delta_{r_0m}(X,\Delta,L)=\inf_{D:r_0m\textrm{-basis}}\mathrm{lct}(X,\Delta;D).
\]
It is known by \cite[Theorem A]{BlJ} that $\lim_{m\to\infty}\delta_{r_0m}(X,\Delta,L)$ exists and we call this the $\delta$-invariant of $(X,\Delta,L)$ denoted by $\delta(X,\Delta,L)$.

On the other hand, let $(X,\Delta,L)$ be a polarized deminormal pair. Let $\nu:\overline{X}\to X$ be the normalization and let $\mathfrak{cond}_{\overline{X}}\subset\overline{X}$ be the conductor subscheme defined by the ideal $\mathcal{H}om_{\mathcal{O}_X}(\nu_*\mathcal{O}_{\overline{X}},\mathcal{O}_X)$. Then $\mathfrak{cond}_{\overline{X}}$ is known to be a reduced Weil divisor \cite[\S5.1]{Ko}. We say $(X,\Delta)$ is {\it slc} if $(\overline{X},\nu_*^{-1}\Delta+\mathfrak{cond}_{\overline{X}})$ is lc and set $\delta(X,\Delta,L)=0$ for non normal polarized slc pairs. 
\end{de}

If $\Delta=\sum_{i=1}^r a_i D_i$ is a $\mathbb{Q}$-divisor where each $D_i$ is an irreducible component of $\Delta$, set 
\[
\chi(\Delta,mH|_{\Delta})=\sum_{i=1}^r a_i \chi(\Delta,mH|_{D_i})
\]
for any line bundle $H$ on $X$ since we are only interested in the leading term of $\chi(\Delta,mH|_{\Delta})$. 

\begin{de}\label{defkst23}
Let $(X,\Delta,L)$ be a deminormal polarized pair. A pair $(\mathcal{X},\mathcal{L})$ is called a {\it semiample test configuration} for $(X,L)$ if the following conditions hold.
\begin{enumerate}
\item $\mathcal{X}$ is a scheme and $\mathcal{L}$ is a semiample $\mathbb{Q}$-line bundle on $\mathcal{X}$ such that $\mathbb{G}_m$ acts on $(\mathcal{X},\mathcal{L})$ in the sense of \cite[\S1.4]{GIT}, 
\item There exists a projective, flat, and $\mathbb{G}_m$-equivariant morphism $\pi:\mathcal{X}\to\mathbb{A}^1$, where $\mathbb{A}^1$ admits a natural $\mathbb{G}_m$-action by multiplication,
\item $(\pi^{-1}(1),\mathcal{L}|_{\pi^{-1}(1)})\cong (X,L)$.
\end{enumerate}
If $\mathcal{L}$ is ample, then we call $(\mathcal{X},\mathcal{L})$ an ample test configuration. For any semiample test configuration $(\mathcal{X},\mathcal{L})$, we get the canonical compactification over $\mathbb{P}^1$ denoted by $(\overline{\mathcal{X}},\overline{\mathcal{L}})$, whose restriction to $\mathbb{P}^1\setminus0$ coincides with $(X\times\mathbb{A}^1,L\times\mathbb{A}^1)$ such that $\mathbb{G}_m$ trivially acts on the first component $X$.

We denote a test configuration $X\times \mathbb{A}^1$, with the trivial $\mathbb{G}_m$-action on the first component, by $X_{\mathbb{A}^1}$. For any ample test configuration $(\mathcal{X},\mathcal{L})$, there exists another semiample test configuration $(\mathcal{Y},\gamma^*\mathcal{L})$ such that there exist two $\mathbb{G}_m$-equivariant morphisms $\gamma:\mathcal{Y}\to\mathcal{X}$ and $\rho:\mathcal{Y}\to X_{\mathbb{A}^1}$. Here, $\gamma$ and $\rho$ induce the identity morphism over $\mathbb{A}^1\setminus 0$. In this case, we say $\mathcal{Y}$ dominates $X_{\mathbb{A}^1}$. We may also assume that $\mathcal{Y}$ is deminormal by \cite[Proposition 3.2]{Fu} and \cite{O3}. Let $\mathcal{D}$ be the closure of $\Delta\times(\mathbb{A}^1\setminus 0)$ in $\overline{\mathcal{X}}$. On the other hand, let $H$ be a $\mathbb{Q}$-line bundle on $X$. Then we define the following functional on $(\mathcal{X},\mathcal{L})$ after \cite{LS}
\begin{align*}
(\mathcal{J}^H)^{\mathrm{NA}}(\mathcal{X},\mathcal{L})=(L^n)^{-1}\left(\overline{\rho^*(H\times\mathbb{A}^1)}\cdot \overline{\gamma^*\mathcal{L}}^n-\frac{nH\cdot L^{n-1}}{(n+1)L^n}\overline{\mathcal{L}}^{n+1}\right).
\end{align*}
We call this the {\it non-Archimedean (nA) J$^H$-functional} of $(X,L)$. It is easy to check that $(\mathcal{J}^H)^{\mathrm{NA}}(\mathcal{X},\mathcal{L})$ does not depend on the choice of $\gamma$. If $\mathcal{X}$ is deminormal, we also define
\[
\mathrm{DF}_{\Delta}(\mathcal{X},\mathcal{L})=(L^n)^{-1}\left((K_{\overline{\mathcal{X}}/\mathbb{P}^1}+\mathcal{D})\cdot \overline{\mathcal{L}}^n-\frac{n(K_X+\Delta)\cdot L^{n-1}}{(n+1)L^n}\overline{\mathcal{L}}^{n+1}\right).
\] 
We call this the (log) {\it Donaldson-Futaki (DF) invariant} of $(\mathcal{X},\mathcal{L})$ (cf., \cite{O}, \cite{W}). 

We define $(X,\Delta,L)$ is
\begin{itemize}
\item uniformly K-stable (resp., K-semistable) if there exists a rational constant $\epsilon>0$ (resp., $\epsilon=0$) such that
\[
\mathrm{DF}_\Delta(\mathcal{X},\mathcal{L})\ge (\mathcal{J}^{\epsilon L})^{\mathrm{NA}}(\mathcal{X},\mathcal{L})
\]
\item uniformly J$^H$-stable (resp., J$^H$-semistable) if there exists a rational constant $\epsilon>0$ (resp., $\epsilon=0$) such that
\[
 (\mathcal{J}^{H})^{\mathrm{NA}}(\mathcal{X},\mathcal{L})\ge (\mathcal{J}^{\epsilon L})^{\mathrm{NA}}(\mathcal{X},\mathcal{L})
\]

\end{itemize}
for any ample deminormal test configuration $(\mathcal{X},\mathcal{L})$ (cf., \cite{BHJ}, \cite{Hat2}). Here, $ (\mathcal{J}^{L})^{\mathrm{NA}}\ge0$ is nothing but the $I^{\mathrm{NA}}-J^{\mathrm{NA}}$-norm in \cite[\S7]{BHJ} or the minimum norm in \cite{De2}. See Lemma \ref{wx} for the proof of the fact that $ (\mathcal{J}^{L})^{\mathrm{NA}}$ is indeed a norm in some sense.
\end{de}

To consider K-stability of polarized deminormal pairs, we may restrict to slc pairs by \cite{Gitod}, \cite[Theorem 6.1]{O4}, \cite[\S9]{BHJ}. 

\subsection{Filtrations and DF invariant}
We assume that for any polarized scheme $(X,L)$, $L$ is $\mathbb{Z}$-Cartier throughout this section.

 \begin{de}\label{fund.def}
 Let $R=\bigoplus_{m\in\mathbb{Z}_{\ge0}} R_m$ be a graded algebra over $\mathbb{C}$ with a unit element $1$.  $\mathscr{F}=\mathscr{F}^{\bullet}R$ is called a {\it linearly bounded multiplicative $\mathbb{Z}$-filtration} of $R$ if $\mathscr{F}$ satisfies the following.
 \begin{enumerate}
 \item For $\lambda>\lambda'\in\mathbb{Z}$ and $m\in\mathbb{Z}$, $\mathscr{F}^{\lambda}R_m\subset\mathscr{F}^{\lambda'}R_m$,
 \item $\mathscr{F}^{\lambda}R_m\cdot \mathscr{F}^{\lambda'}R_{m'}\subset\mathscr{F}^{\lambda+\lambda'}R_{m+m'}$ for $\lambda,\lambda'\in\mathbb{Z}$ and $m,m'\in\mathbb{Z}_{\ge0}$,
 \item There exists a positive constant $C$ such that for sufficiently large $m\in\mathbb{Z}_{\ge0}$, $\mathscr{F}^{\lambda}R_m=0$ for $\lambda\ge Cm$ and $\mathscr{F}^{\lambda}R_m=R_m$ for $\lambda< -Cm$,
 \item $1\in\mathscr{F}^\lambda R_0$ for any $\lambda$.
 \end{enumerate}
In this paper, we call $\mathscr{F}$ a filtration for simplicity. Moreover, if $\bigoplus_{m\in\mathbb{Z}_{\ge0},\lambda\in\mathbb{Z}}\mathscr{F}^{\lambda}R_m$ forms a finitely generated bigraded $\mathbb{C}$-algebra, then we say $\mathscr{F}$ is finitely generated. 
 
Suppose that $\mathrm{dim}\, R_m$ is finite. We define $w_{\mathscr{F}}(m)$ the {\it weight function} of $\mathscr{F}$ as
 \[
 w_{\mathscr{F}}(m)=\sum_{\lambda\in\mathbb{Z}}\lambda\,\mathrm{dim}\,(\mathscr{F}^{\lambda}R_m/\mathscr{F}^{\lambda+1}R_m).
 \] 
  Let $N\in\mathbb{Z}$. The {\it weight $N$-shift} $ \mathscr{F}_{(N\textrm{-shift})}$ of $\mathscr{F}$ is a filtration defined by
 \[
 \mathscr{F}_{(N\textrm{-shift})}^{\lambda}R_k= \mathscr{F}^{\lambda+Nk}R_k.
 \]
 This indeed satisfies the conditions of Definition \ref{fund.def}. 
 
 For $x\in\mathbb{R}$, we define $\mathscr{F}^{x}R_m=\mathscr{F}^{\lceil x\rceil}R_m$ and let $R^{(x)}=\bigoplus_{m\ge0}R^{(x)}_m=\bigoplus_{m\ge0}\mathscr{F}^{mx}R_m$. Then $R^{(x)}$ is a graded subalgebra of $R$ and it holds that
  \begin{equation}\label{eqint}
 w_{\mathscr{F}}(m)=\int^{\infty}_{-Cm}\mathrm{dim}\, \mathscr{F}^{\lambda}R_md\lambda-Cm\mathrm{dim}\, R_m
 \end{equation}
 if $m$ and $C$ satisfy the condition (3) above (see \cite[\S5]{BHJ} or \cite[Prop. 2.12 (2)]{Fjt}).
 \end{de}
 \begin{ex}
 There exists the trivial filtration $\mathscr{F}_{\mathrm{triv}}$, which is defined by $\mathscr{F}_{\mathrm{triv}}^\lambda R_k=R_k$ if $\lambda\le0$ or $k=0$, otherwise $\mathscr{F}_{\mathrm{triv}}^\lambda R_k=0$.
 \end{ex}


The following is the most important case in this paper. If $R$ is a graded subalgebra of $\bigoplus_{m\ge0} H^0(X,mL)$ for some polarized scheme $(X,L)$, then we call $R$ a {\it linear series} of $(X,L)$. We define the {\it volume} of $R$ as
 \[
 \mathrm{vol}(R)=n!\limsup_{m\to\infty}\frac{\mathrm{dim}\,R_m}{m^n}.
 \]
 If $\mathscr{F}$ is further a filtration of $R=\bigoplus_{m\ge0} H^0(X,mL)$, then we call $\mathscr{F}$ a {\it filtration} of $(X,L)$.   
 \begin{ex}\label{dresf}
 Let $(X,L)$ be a polarized scheme and $D$ be a closed subscheme. If $\mathscr{F}$ is a filtration of $R=\bigoplus_{m\ge0}H^0(X,mL)$, then we define a new filtration $\mathscr{F}_D$ of $\bigoplus_{m\ge0}H^0(D,mL|_D)$ by
  \[
\mathscr{F}_D^{\lambda}H^0(D,mL|_D)=\mathrm{Im}\left(\mathscr{F}^{\lambda}H^0(X,mL)\to H^0(D,mL|_D)\right).
 \] We call $\mathscr{F}_D$ the {\it restriction} of $\mathscr{F}$ to $D$ or the {\it induced} filtration. We also denote $R^{(\lambda)}|_{D}=\bigoplus_{m\ge0}(R^{(\lambda)}|_{D})_m=\bigoplus_{m\ge0}\mathscr{F}_D^{\lambda}H^0(D,mL|_D)$. It is easy to see that $\mathscr{F}_D$ is linearly bounded and multiplicative by the Serre vanishing theorem.
 \end{ex}
 
  \subsubsection{Volumes of linear series on varieties}\label{vol}
  
  We recall the results on volumes of linear series of varieties (cf., \cite{BC}). Let $(X,L)$ be an {\bf integral} polarized variety and $R$ be a linear series of $(X,L)$. We say $R$ {\it contains an ample series} if the following hold.
  \begin{enumerate}
  \item $R_k\ne0$ for sufficiently large $k>0$,
  \item There exist an integer $m\in\mathbb{Z}_{>0}$ and an ample line bundle $A_m$ such that $mL-A_m$ is effective and 
  \[
  H^0(X,A_m)\subset R_m\subset H^0(X,mL).
  \]
  \end{enumerate}
If (1) and (2) hold for some $m$, it is known that then (2) hold for any $m>0$ (see \cite[2.10]{LM}). 

Recall the definition of the Okounkov body of $R$ in \cite{LM} and \cite[\S1]{BC}. Suppose that $R$ contains an ample series. Fix a smooth closed point $x\in X$ and a regular system of parameters $(z_1,\cdots,z_n)$ at $x$. Let $\mathrm{ord}_x$ be the canonical valuation of rational rank $n$ defined by the lexicographic order with respect to $(z_1,\cdots,z_n)$. We have the following map with the image whose cardinality is $\mathrm{dim}\, R_m$
\[
\frac{1}{m}\mathrm{ord}_x:R_m\setminus0\to\mathbb{Q}^n_{\ge0}.
\]
Let $\Delta_R$ be the closed convex hull of $\bigcup_{m\ge0}\frac{1}{m}\mathrm{ord}_x(R_m\setminus0)$ and we call $\Delta_R$ the {\it Okounkov body} of $R$. It is known that $\Delta_R$ is bounded and if $\rho$ is the Lebesgue measure on $\mathbb{R}^n$, then $\mathrm{vol}(\Delta_R)=\int_{\Delta_R}d\rho=\frac{1}{n!}\mathrm{vol}(R)$ by \cite[1.12]{B1}. It follows from this fact that $\lim_{m\to\infty}\frac{\mathrm{dim}R_m}{m^n}$ exists. 


  Next, we consider when $R=\bigoplus_{m\ge0}H^0(X,mL)$ and $\mathscr{F}$ is a filtration of $R$. Set $$e_{\max}(R,\mathscr{F})=\limsup_{k\to \infty}\frac{\sup\{t\in\mathbb{R}|\mathscr{F}^tR_k\ne0\}}{k}$$ and then $R^{(t)}$ contains an ample series for $t<e_{\max}(R,\mathscr{F})$ by \cite[Lemma 1.6]{BC}. On the other hand, it is easy to see that $\mathrm{vol}(R^{(t)})=0$ for $t>e_{\max}(R,\mathscr{F})$. Then set $\Delta_t=\Delta_{R^{(t)}}$ for $t<e_{\max}(R,\mathscr{F})$. For $t<s<e_{\max}(R,\mathscr{F})$, we have $\Delta_s\subset\Delta_t$. Set $\Delta=\Delta_R$. Then we define the {\it concave transformation} $G:\Delta\to\mathbb{R}$ {\it associated with} $\mathscr{F}$ by $$G(p)=\sup\{t\in\mathbb{R}|p\in\Delta_t\}$$ for $p\in\Delta$. It is easy to see that $G$ is concave and upper semicontinuous (cf., \cite{BlJ}). 
   \begin{rem}\label{before}
 Our notation of linearly bounded multiplicative $\mathbb{Z}$-filtrations is different from one of Sz\'{e}kelyhidi \cite{Sz} in sign. 
  
 \end{rem}
 
 Let $w$ be the weight function of $\mathscr{F}$ and take a constant $C$ such that $R_m^{(-C)}=R_m$ for sufficiently large $m\in\mathbb{Z}_{>0}$. Then it follows from 
the equation (\ref{eqint})
for such $m$ and from the dominated convergence theorem that 
\begin{equation}\label{eqlimit}
n! \lim_{m\to\infty}\frac{w(m)}{m^{n+1}}=\int^{\infty}_{-C}\mathrm{vol}R^{(x)}dx-C\mathrm{vol}(R).
\end{equation}
 On the other hand, it follows from \cite[Theorem 1.11]{BC} that
 \begin{equation}\label{eqbary}
  \lim_{m\to\infty}\frac{w(m)}{m^{n+1}}=\int_{\Delta}Gd\rho.
 \end{equation}
In \cite[\S5]{BHJ}, $\frac{1}{\mathrm{vol}(\Delta)}\left(\int_\Delta Gd\rho\right)$ is called the {\it barycenter} of (the Duistermaat-Heckman measure associated with) $\mathscr{F}$. Taking Remark \ref{before} into account, we define the norm of $\mathscr{F}$ as \cite{Sz}
 \[
\|\mathscr{F}\|_2=\sqrt{\int_\Delta G^2d\rho-\frac{1}{\mathrm{vol}(\Delta)}\left(\int_\Delta Gd\rho\right)^2}.
\]
   \subsubsection{The weights of filtrations of polarized reducible or non-reduced schemes}\label{nonreduced}
 As we saw in \S\ref{vol}, the theory of Okounkov bodies is useful to calculate the volumes of linear series of varieties but we can not apply this to general polarized schemes directly. In \S\ref{3}, we consider families over curves and compare the intersection numbers with the weight functions of filtrations of central fibers. However, such fibers may be reducible or non-reduced in general. In this subsection, we discuss the weight functions of filtered linear series of reducible or non-reduced schemes.
 
Let $(X,L)$ be a polarized scheme, $R_m=H^0(X,mL)$ and $\mathscr{F}$ be a linearly bounded multiplicative filtration of $R=\bigoplus_{m\ge0} R_m$. Fix constants $C>0$ and $m_0\in\mathbb{Z}_{>0}$ such that $R_{m}^{(-C)}=R_{m}$ for $m\ge m_0$. 
 As \S \ref{vol}, we define the {\it barycenter} of $\mathscr{F}$ to be
\[
B_{\mathscr{F}}=\limsup_{m\to\infty}\frac{w_{\mathscr{F}}(m)}{m^{n+1}}.
\]
 We use notations as in Example \ref{dresf} and show the following to deduce Lemma \ref{cl3} below.
\begin{prop}\label{restr}
 Let $\{X_i\}_{i=1}^r$ be the set of irreducible components of $X$. We define the scheme structure of $X_i$ by the image structure of its dense open subset. Let $m_i$ be the multiplicity of $X_i$, i.e., the length of $\mathcal{O}_{X,\eta_i}$ where $\eta_i$ is the generic point of $X_i$. Let also $X_{i,\mathrm{red}}$ be the reduced structure of $X_i$ and $\mathscr{F}_{i,\mathrm{red}}=\mathscr{F}_{X_{i,\mathrm{red}}}$.

Then $B_{\mathscr{F}}\ge \sum_{i=1}^rm_iB_{\mathscr{F}_{i,\mathrm{red}}}$.
\end{prop}

 Set $$\underline{\mathrm{vol}}(R)=n!\liminf_{m\to\infty}\frac{\mathrm{dim}\,R_m}{m^n}$$ and $e_i=e_{\max}(R|_{X_{i,\mathrm{red}}},\mathscr{F}_{X_{i,\mathrm{red}}})$. Proposition \ref{restr} follows from the lemma below.

 \begin{lem}\label{rest1}
For $t\in\mathbb{R}\setminus\{e_1,\ldots,e_r\}$, $\underline{\mathrm{vol}}(R^{(t)})\ge \sum_{i=1}^rm_i\mathrm{vol}(R^{(t)}|_{X_{i,\mathrm{red}}})$. 
 \end{lem}

\begin{proof}[Proof of Proposition \ref{restr}]
We assume Lemma \ref{rest1}. Then it follows from Fatou's lemma and from the equation (\ref{eqint}) applied to $\mathscr{F}$ that
\begin{align*}
n!B_{\mathscr{F}}&\ge n!\liminf_{m\to\infty}\frac{w_{\mathscr{F}}(m)}{m^{n+1}}\ge\int^{\infty}_{-C}\underline{\mathrm{vol}}R^{(x)}dx-C\mathrm{vol}(R)\\
&\ge \sum_{i=1}^rm_i\left(\int^{\infty}_{-C}\mathrm{vol}(R^{(t)}|_{X_{i,\mathrm{red}}})dx-C\mathrm{vol}(R|_{X_{i,\mathrm{red}}})\right)=n!\sum_{i=1}^rm_iB_{\mathscr{F}_{i,\mathrm{red}}}.
\end{align*}
The last equality follows from the equation (\ref{eqlimit}).
\end{proof}
Thus it suffices to show Lemma \ref{rest1}. First, note that it suffices to consider the case when the canonical morphism $\mathcal{O}_X\to \prod \mathcal{O}_{X_i}$ is injective by replacing $X$ by the closed subscheme defined by the ideal $\mathrm{Ker}(\mathcal{O}_X\to \prod \mathcal{O}_{X_i})$. Let $\mathfrak{c}\subset\mathcal{O}_X$ be the inverse image of $\mathrm{Hom}_{\mathcal{O}_X}(\prod \mathcal{O}_{X_i},\mathcal{O}_X)$ under the natural map $\mathcal{O}_X\to \mathrm{Hom}_{\mathcal{O}_X}(\prod \mathcal{O}_{X_i},\prod \mathcal{O}_{X_i})$. 

 \begin{lem}\label{redble}
Suppose that $t\ne e_i$ for $1\le i\le r$. Then $\underline{\mathrm{vol}}(R^{(t)})\ge\sum_{i=1}^r\underline{\mathrm{vol}}(R^{(t)}|_{X_i})$. 
  \end{lem}
 
 \begin{proof}
 First, we prove that if $t<e_i$, then there exist an integer $m_i\in\mathbb{Z}_{>0}$ and $s_i\in R_{m_i}^{(t)}\cap H^0(X,m_iL\otimes\mathfrak{c}\cdot \mathrm{Ann}(\mathscr{I}_{X_i}))$ such that $s_i$ is a unit at the generic point $\eta_i$ of $X_i$ where $\mathrm{Ann}(\mathscr{I}_{X_i})$ is the annihilator of the ideal corresponding to $X_i$. Indeed, there exists a section $s'\in H^0(X,lL\otimes\mathfrak{c}\cdot \mathrm{Ann}(\mathscr{I}_{X_i}))$ such that $s'$ is a unit at $\eta_i$ for some $l\in\mathbb{Z}_{>0}$. Consider $s'\in R_{l}^{(-C)}$. Next, take a sufficiently small constant $0<\epsilon\ll1$ that $t+\epsilon<e_i$. Then $s'\cdot R_{p}^{(t+\epsilon)}\subset R_{l+p}^{(\frac{p(t+\epsilon)-Cl}{l+p})}$. If we take $p$ sufficiently large that $p\epsilon\ge (C+t)l$, then $s'\cdot R_{p}^{(t+\epsilon)}\subset R_{l+p}^{(t)}$. Furthermore, there exists a section $s''\in R_{p}^{(t+\epsilon)}$ such that the restriction $s''|_{X_{i,\mathrm{red}}}\ne0$
 . Thus, let $s_i=s's''$ and $m_i=l+p$.
 
 Finally, let $0\le r'\le r$ be the integer such that $t<e_i$ iff $i\le r'$. By what we have shown, there exist sections $s_i\in R_{m_i}^{(t)}\cap H^0(X,m_iL\otimes\mathfrak{c})$ such that $s_i$ is a unit at $\eta_i$ for any $i\le r'$. Replacing $s_i$ by $s_i^k$ for some $k\in\mathbb{Z}_{>0}$, we may assume that $m=m_i$ for $i\le r'$. Consider the following $\mathbb{C}$-linear map
 \[
h: \prod_{i\le r'}(R|_{X_i})_k\ni(t_i|_{X_i})\mapsto \sum_{i\le r'} s_it_i\in R_{k+m}.
 \]
 It is easy to see that $h$ is well-defined and $\mathrm{Ker}\, h\subset \bigoplus H^0(X_i,mL|_{X_i}\otimes \mathrm{Ker}\, h')$ where $h':\prod_{i\le r'}\mathcal{O}_{X_i}\to \mathcal{O}_X$ is the map induced by $s_i$'s. Since $\mathrm{Ker}\, h'$ has the nowhere-dense support in $\bigcup_{i\le r'} X_i$, we have for $i\le r'$, $$\lim_{m\to\infty}\frac{h^0(X_i,mL|_{X_i}\otimes\mathrm{Ker}\, h')}{m^n}=0.$$ Thus we have $\underline{\mathrm{vol}}(R^{(t)})\ge\sum_{i=1}^{r'}\underline{\mathrm{vol}}(R^{(t)}|_{X_i})$.
   \end{proof}
   \begin{rem}
 For general linear series $R$, $\underline{\mathrm{vol}}(R)\ge\sum_{i=1}^r\underline{\mathrm{vol}}(R|_{X_i})$ does not hold. Fix a closed point $0\in\mathbb{P}^1$. Let $X=\mathbb{P}^1\cup_{0}\mathbb{P}^1$ be a reducible curve with two irreducible components $\mathbb{P}^1$ intersecting transversally at $0$. Let $R_m$ be the diagonal of $H^0(\mathbb{P}^1,\mathcal{O}(m))\oplus H^0(\mathbb{P}^1,\mathcal{O}(m))$. Then let $R=\bigoplus_{m\ge0}R_m$ and we have $\underline{\mathrm{vol}}(R)=\mathrm{vol}(R|_{\mathbb{P}^1})=1$ for two components.
  \end{rem}

  By Lemma \ref{redble}, it suffices to show Lemma \ref{rest1} when $X$ is irreducible.
\begin{proof}[Proof of Lemma \ref{rest1}]
We may assume that $X$ is irreducible and $t<e_{\max}(R|_{X_{\mathrm{red}}},\mathscr{F}_{X_{\mathrm{red}}})$. Let $m_0$ be the multiplicity of $X$. We prove the assertion by the induction on $m_0$. Suppose that $m_0>1$ and let $\mathcal{O}_{X,\eta}$ be the local ring at the generic point $\eta$. Since $\mathcal{O}_{X,\eta}$ is Artinian, there exists an element $f\in\mathcal{O}_{X,\eta}$ that generates a non-zero minimal ideal. Here, we fix an isomorphism of $L|_U$ and $\mathcal{O}_U$ where $U$ is a non-empty open subset. Since $L$ is ample, there exists a section $s\in H^0(X,mL)$ such that the germ $s_{\eta}$ at $\eta$ of $s$ generates the minimal ideal $f\cdot\mathcal{O}_{X,\eta}$ for sufficiently large $m>0$. As in the proof of Lemma \ref{redble}, we may replace $s$ and $m$, and assume that $s\in R^{(t)}_m$ and generates $f\cdot\mathcal{O}_{X,\eta}$. For $p\in\mathbb{Z}_{\ge0}$, we consider a surjective map $r_p:
R_{p}^{(t)}\to (R^{(t)}|_{X'})_{p}$
where $X'$ is the closed subscheme defined by the ideal generated by $s$. Since the multiplicity of $X'$ is $m_0-1$, it follows that $$\underline{\mathrm{vol}}(R^{(t)}|_{X'})\ge (m_0-1)\mathrm{vol}(R^{(t)}|_{X_{\mathrm{red}}})$$ from the induction hypothesis. Note that $\mathrm{Ker}\, r_{p+m}$ contains $s\cdot R^{(t)}_p$. Let $X''$ be the closed subscheme defined by $\mathrm{Ann}(s)$. Then 
\[
s\cdot R^{(t)}_p\cong (R^{(t)}|_{X''})_p.
\]
Thus we have $\underline{\mathrm{vol}}(R^{(t)})\ge\underline{\mathrm{vol}}(R^{(t)}|_{X'})+\underline{\mathrm{vol}}(R^{(t)}|_{X''})$. Since $X''$ is generically reduced, $$\underline{\mathrm{vol}}(R^{(t)}|_{X''})\ge \mathrm{vol}(R^{(t)}|_{X_{\mathrm{red}}}).$$
Hence the assertion holds.
\end{proof}

On the other hand, we remark that if $X$ is reduced then we have equality in Proposition \ref{restr}. We first observe that if $X$ is reduced in Lemma \ref{redble}, then
  \begin{equation}
  \mathrm{vol}(R^{(t)})=\underline{\mathrm{vol}}(R^{(t)})=\sum_{i=1}^{r}\mathrm{vol}(R^{(t)}|_{X_i}).\label{eqlast}
  \end{equation}  Indeed, by the restriction map $R^{(t)}\to \prod_{i=1}^r R^{(t)}|_{X_i}$, we have
  \begin{equation*}
  \mathrm{vol}(R^{(t)})\le\sum_{i=1}^{r}\mathrm{vol}(R^{(t)}|_{X_i}).
  \end{equation*}
  \begin{lem}\label{limred}
 Notations as above. Then $\lim_{l\to \infty}\frac{w(l)}{l^{n+1}}$ exists and we have
\[
\lim_{l\to \infty}\frac{w(l)}{l^{n+1}}=B_{\mathscr{F}}= \sum_{i=1}^rB_{\mathscr{F}_{i,\mathrm{red}}}.
\]
  \end{lem}
  \begin{proof}
  Note that the equation (\ref{eqlast}) holds
  for $x\in \mathbb{R}\setminus\{e_1,\ldots,e_r\}$. Thus, we obtain the assertion by the dominated convergence theorem applied to the equation (\ref{eqint}).
  \end{proof}
   \begin{rem}
   $\underline{\mathrm{vol}}(R^{(t)})\ge\sum_{i=1}^{r'}m_i\mathrm{vol}(R^{(t)}|_{X_i})$ can be strict in general. Consider $X=\mathbb{P}^1\times_\mathbb{C}\mathrm{Spec}\, \mathbb{C}[\epsilon]/(\epsilon^2)$, $L=\mathcal{O}(1)$ and 
  \[
  \mathscr{F}^{\lambda}R_{m}=H^0(\mathbb{P}^1,\mathcal{O}(m-\lambda))\oplus\epsilon \mathscr{F}^{\lambda-m}_{\mathrm{triv}}H^0(\mathbb{P}^1,\mathcal{O}(m)).
  \]
  \end{rem}

\subsubsection{The Donaldson-Futaki invariants of good filtrations}\label{defgood}

In this section, we define good filtrations and the DF invariants of them.
 \begin{de}\label{goodfil}
 Let $(X,L)$ be an $n$-dimensional polarized deminormal scheme and $\mathscr{F}$ be a linearly bounded multiplicative $\mathbb{Z}$-filtration of $R=\bigoplus_{m\ge0}H^0(X,mL)$. 

Let $w(r)$ be the weight function of $\mathscr{F}$. Suppose that $w(r)=b_0r^{n+1}+b_1r^n+O(r^{n-1})$. Then we call $\mathscr{F}$ a {\it good} filtration of $R$ and we define the {\it DF invariant} of $\mathscr{F}$ as
\[
\mathrm{DF}(\mathscr{F})=2\frac{b_0a_1-b_1a_0}{a_0^2},
\]
where $\chi(X,kL)=a_0k^{n}+a_1k^{n-1}+O(k^{n-2})$. On the other hand, we define the $r$-th Chow weight as
 \[
 \mathrm{Chow}_r(\mathscr{F})=2\left(\frac{rb_0}{a_0}-\frac{w(r)}{\chi(X,rL)}\right).
 \]

 If $(X,\Delta,L)$ is a deminormal polarized pair, then we define the log DF invariant of a good filtration $\mathscr{F}$ as follows. If $\Delta=\sum c_iD_i$ and $\mathscr{F}_{D_i}$ is the restriction of $\mathscr{F}$ to $D_i$ with the weight function $w_i$, then we set 
   \[
   \mathrm{DF}_{\Delta}(\mathscr{F})=\mathrm{DF}(\mathscr{F})-\frac{b_0\tilde{a}_0-\tilde{b}_0a_0}{a_0^2}.
   \]
   Here, $\chi(\Delta,mL)=\tilde{a}_0m^{n-1}+O(m^{n-2})$ and $\tilde{b}_0=\lim_{m\to\infty}\sum \frac{c_iw_i(m)}{m^n}$.
 \end{de}

Next, we prepare the useful condition below.
\begin{ass}\label{crucial}
Let $(X,L)$ be a polarized {\bf reduced} scheme and $X=\bigcup_{i=1}^rX_i$ be the irreducible decomposition. Here, let $R=\bigoplus_{m\ge0}H^0(X,mL)$ and assume that $R|_{X_i}=\bigoplus_{m\ge0}H^0(X_i,mL|_{X_i})$ holds. Assume also that $H^0(X,L)$ generates $R$ and $R|_{X_i}$ for $i$. If $\mathscr{F}$ is a filtration of $(X,L)$, assume that there exists $N\in\mathbb{Z}_{>0}$ such that $R^{(-N)}=R$.
\end{ass}
Condition \ref{crucial} is also assumed in \cite[\S3]{Sz} to define an approximation in Definition \ref{appro} below. 
\begin{de}\label{appro}
Under Condition \ref{crucial}, take sufficiently large $N\in\mathbb{Z}_{>0}$ that $R^{(-N)}=R$. Suppose that $\{\mathscr{F}_{(k)}\}_{k\in\mathbb{Z}_{>0}}$ is a sequence of finitely generated filtrations of $R$ generated by $\mathscr{F}^{\bullet}R_k$ and $\mathscr{F}_{\mathrm{triv},(-N\textrm{-shift})}^\bullet R$ as in \cite[\S 3.2]{Sz} for $k\in\mathbb{Z}_{>0}$. Note that $\mathscr{F}_{(k)}^\lambda R_m\subset \mathscr{F}^\lambda R_m$. Then we call $\{\mathscr{F}_{(k)}\}_{k\in\mathbb{Z}_{>0}}$ an {\it approximation} to $\mathscr{F}$. 
 \end{de}
 
We remark that $\lim_{k\to\infty}B_{\mathscr{F}_{(k)}}$ (cf., \S\ref{nonreduced}) is independent of the choice of $N$.
 
\begin{lem}\label{restrr}
 Notations as above. Then  
\[
B_{\mathscr{F}}=\lim_{k\to \infty}B_{\mathscr{F}_{(k)}}.
\]
\end{lem}

\begin{proof}
The assertion when $X$ is irreducible follows from \cite[Lemma 6]{Sz} and the equation (\ref{eqbary}). Thus the assertion follows in the general case from Lemma \ref{limred}.
\end{proof}

 For any reduced closed subscheme $D\subset X$, replacing $L$ by $cL$ for some $c\in\mathbb{Z}_{>0}$, we may assume that $\mathscr{F}_D$ in Example \ref{dresf} satisfies Condition \ref{crucial} by the Serre vanishing theorem. Furthermore, we may assume that $\{(\mathscr{F}_{(k)})_D\}_{k\in\mathbb{Z}_{>0}}$ is an approximation to $\mathscr{F}_D$.

 We remark the following important result of Sz\'{e}kelyhidi that we make use of in the proof of Corollary \ref{cmcsck}.
\begin{thm}[\cite{Dn}, \cite{Sz}]\label{prop11}\label{ch}
Let $(X,L)$ be a polarized smooth variety with a cscK metric in $\mathrm{c}_1(L)$ 
such that $\mathrm{Aut}(X,L)$ discrete. If $\mathscr{F}$ is a good filtration of $(X,L)$, then
\[
\mathrm{DF}(\mathscr{F})\ge0.
\]
Furthermore, if $\|\mathscr{F}\|_2>0$, then $\mathrm{DF}(\mathscr{F})>0.$
\end{thm}

 \begin{proof}
First, we may replace $L$ by $cL$ for some $c\in\mathbb{Z}_{>0}$ and assume that $(X,L)$ and $\mathscr{F}$ satisfy Condition \ref{crucial}. Take an approximation $\{\mathscr{F}_{(r)}\}_{r\in\mathbb{Z}_{>0}}$ to $\mathscr{F}$. Then we define 
\[
 \mathrm{Chow}_\infty(\mathscr{F})=\liminf_{r\to\infty} \mathrm{Chow}_r(\mathscr{F}_{(r)})
\]
after \cite{Sz}.
 Let $w(k)=b_0k^{n+1}+b_1k^n+O(k^{n-1})$ and $w_r(k)=b_0^{(r)}k^{n+1}+O(k^n)$ be the weights of $\mathscr{F}^\bullet R_k$ and $\mathscr{F}_{(r)}^\bullet R_k$ respectively. If
$
 \mathrm{DF}(\mathscr{F})\ge \mathrm{Chow}_{\infty}(\mathscr{F})
$
 holds, then the assertion follows from \cite[Corollary 4]{Dn} and from \cite[Proposition 11]{Sz}. 
 
 For this, we show $b_0^{(r)}\le b_0$. Let $\Delta$ be the Okounkov body of $(X,L)$ and $\rho$ be the Lebesgue measure of $\Delta$. Let also $G$ (resp., $G^{(r)}$) be the concave transformation with respect to $\mathscr{F}$ (resp., $\mathscr{F}^{(r)}$). By the equation (\ref{eqbary}), \[b_0=\int_\Delta G\,d\rho,\,\,b_0^{(r)}=\int_\Delta G^{(r)}d\rho\] and $G\ge G^{(r)}$ (cf., \cite[Lemma 6]{Sz}, Remark \ref{before}). Thus we have $b_0^{(r)}\le b_0$.
 
 On the other hand, $w_r(r)=w(r)$ since $\mathscr{F}^\bullet R_r=\mathscr{F}_{(r)}^\bullet R_r$. Thus, we have
 \begin{align*}
 \mathrm{DF}(\mathscr{F})=\lim_{r\to\infty}2\left(\frac{rb_0}{a_0}-\frac{w(r)}{\chi(X,rL)}\right)\ge\liminf_{r\to\infty}2\left(\frac{rb_0^{(r)}}{a_0}-\frac{w_r(r)}{\chi(X,rL)}\right)= \mathrm{Chow}_{\infty}(\mathscr{F}).
 \end{align*}
 We complete the proof.
  \end{proof}


Finally, we introduce the nA J-functionals of filtrations as follows.
 \begin{de}\label{compatible.divisor}
 Let $(X,L)$ be a polarized deminormal scheme, $\mathscr{F}$ be a linearly bounded multiplicative filtration of $(X,L)$ and $H$ be an ample $\mathbb{Q}$-line bundle on $X$. Note that we define $(\mathcal{J}^H)^{\mathrm{NA}}(\mathscr{F})$ by using a very general divisor $D$ as the equation (\ref{jdef}) below. Taking Lemma \ref{limred} into account, we may replace $L$ by $cL$ for some $c\in\mathbb{Z}_{>0}$ and assume that $\mathscr{F}$ and $(X,L)$ satisfy Condition \ref{crucial} to define $(\mathcal{J}^H)^{\mathrm{NA}}(\mathscr{F})$. Take an approximation $\{\mathscr{F}_{(k)}\}_{k\in\mathbb{Z}_{>0}}$ to $\mathscr{F}$. We define a semiample test configuration $(\mathcal{X}^{(k)},\mathcal{L}^{(k)})$ that dominates $X_{\mathbb{A}^1}$ as follows. Let $\mathfrak{a}_{(k)}$ be the image of the following evaluation map where $t$ is the canonical coordinate of $\mathbb{A}^1$
 \[
 \bigoplus_{\lambda} t^{-\lambda}\mathscr{F}^\lambda H^0(X,kL)\otimes\mathcal{O}_{X_{\mathbb{A}^1}}(-kL\times\mathbb{A}^1)\to \mathcal{O}_{X}[t,t^{-1}].
 \]
 This is a $\mathbb{G}_m$-invariant fractional ideal of $\mathcal{O}_{X_{\mathbb{A}^1}}$ and called a {\it flag ideal} (cf., \cite[3.1]{O3}, \cite[\S2.6]{BHJ}). Then, let $\mu_k:\mathcal{X}^{(k)}\to X_{\mathbb{A}^1}$ be the blow up along $\mathfrak{a}_{(k)}$ and $$\mathcal{L}^{(k)}=\mu_k^*(L\times\mathbb{A}^1)-\mu_k^{-1}(\mathfrak{a}_{(k)}).$$
 
 For sufficiently divisible $m>0$, $mH$ is a very ample $\mathbb{Z}$-line bundle and there exists a non-empty open subset consisting of $D\in|mH|$ such that the support of $\mu_k^*D_{\mathbb{A}^1}$ contains no $\mu_k$-exceptional divisor. We take $D$ so general that 
 $\mu_k^*D'_{\mathbb{A}^1}$ contains no $\mu_k$-exceptional divisor of $\mathcal{X}^{(k)}$. We call such $D$ compatible with $\mathscr{F}_{(k)}$. Since $\mathbb{C}$ is uncountable, there exists an effective $\mathbb{Q}$-Cartier divisor $D\sim_{\mathbb{Q}} H$ such that $D$ is compatible with any $\mathscr{F}_{(k)}$
 . We call such $D$ {\it compatible} with $\{\mathscr{F}_{(k)}\}_{k\in\mathbb{Z}_{>0}}$. On the other hand, set $$(\mathcal{J}^H)^{\mathrm{NA}}(\mathscr{F}_{(k)})=(\mathcal{J}^H)^{\mathrm{NA}}(\mathcal{X}^{(k)},\mathcal{L}^{(k)}).$$
 
 Next, take $D$ a compatible divisor with $\{\mathscr{F}_{(k)}\}_{k\in\mathbb{Z}_{>0}}$. We know by Lemma \ref{limred} that there exists a constant $\tilde{b}_{0,i}$ for any irreducible component $D_i$ of $D$ such that
 \[
 \tilde{b}_{0,i}=\lim_{m\to\infty}\frac{\tilde{w}_{\mathscr{F}_{D_i}}(m)}{m^n}
 \]
(cf., Example \ref{dresf}). Then we define the {\it nA J$^H$-functional} of $\mathscr{F}$ as
 \begin{equation}
(\mathcal{J}^H)^{\mathrm{NA}}(\mathscr{F})=\frac{\tilde{b}_0a_0-b_0\tilde{a}_0}{a_0^2}.\label{jdef}
 \end{equation}
 Here, $\tilde{a}_0=\lim_{m\to\infty}\frac{\chi(D,mL)}{m^{n-1}}$ and $\tilde{b}_0=\sum m_i\tilde{b}_{0,i}$ where $D=\sum m_iD_i$.
 \end{de}

 \begin{lem}\label{lim}
 Notations as above. If $D\sim_{\mathbb{Q}}H$ is compatible with $\{\mathscr{F}_{(k)}\}_{k\in\mathbb{Z}_{>0}}$
 , then 
 \begin{equation}\label{limjlim}
 \lim_{k\to\infty}(\mathcal{J}^H)^{\mathrm{NA}}(\mathscr{F}_{(k)})=(\mathcal{J}^H)^{\mathrm{NA}}(\mathscr{F}).
 \end{equation}
 In particular, the following hold.
 \begin{enumerate}[(i)]
 \item $(\mathcal{J}^H)^{\mathrm{NA}}(\mathscr{F})$ is independent from the choice of a compatible divisor $D$.
 \item If $(X,L)$ is {\rm J}$^H$-semistable, then $(\mathcal{J}^H)^{\mathrm{NA}}(\mathscr{F})\ge0$ for any filtration.
 \end{enumerate}
 \end{lem}
 
 \begin{proof}
 Note that it immediately follows from the equation (\ref{limjlim}) that (i) and (ii) hold. It suffices to show the first assertion.
 
Let $D=\sum m_iD_i$ be the irreducible decomposition and note that $D_i$ is compatible with any $\mathscr{F}_{(k)}$. Let $\tilde{w}_i^{(k)}(m)=\tilde{b}^{(k)}_{0,i}m^n+O(m^{n-1})$ be the weight function of $(\mathscr{F}^\bullet_{(k)})_{D_i}$. On the other hand, let $w^{(k)}(m)=b_0^{(k)}m^{n+1}+O(m^n)$ be the weight of $\mathscr{F}^\bullet_{(k)}H^0(X,mL)$. Then we claim the following hold.
\begin{itemize}
\item[(a)]
$
(\mathcal{J}^H)^{\mathrm{NA}}(\mathscr{F}_{(k)})=-\frac{b^{(k)}_0\tilde{a}_0-\sum m_i\tilde{b}^{(k)}_{0,i}a_0}{a_0^2},
$
\item[(b)] $\lim_{k\to\infty}b_0^{(k)}=b_0$ and $\lim_{k\to\infty}\sum m_i\tilde{b}^{(k)}_{0,i}=\tilde{b}_0$. 
\end{itemize}
Indeed, we may assume that $\{(\mathscr{F}^\bullet_{(k)})_{D_i}\}_{k\in\mathbb{Z}_{>0}}$ is an approximation to $\mathscr{F}_{D_i}$ by replacing $L$ by $lL$ for $l\in\mathbb{Z}_{>0}$ sufficiently divisible
. Thus, (b) follows from Lemma \ref{restrr}. To see (a), we may assume that $X$ is integral, $D=D_1$ and $m_1=1$
. If necessary, we assume that $\mathscr{F}^{i}H^0(X,mL)=0$ for $i>0$ by replacing $\mathscr{F}$ by $\mathscr{F}_{(N\textrm{-shift})}$ for a suitable $N\in\mathbb{Z}$. Note that $(\mathcal{J}^H)^{\mathrm{NA}}(\mathscr{F})=(\mathcal{J}^H)^{\mathrm{NA}}(\mathscr{F}_{(N\textrm{-shift})})$. Let $\mathfrak{a}_{(k)}$ be the flag ideal and $(\mathcal{X}^{(k)},\mathcal{L}^{(k)})$ be the induced test configuration as in Definition \ref{compatible.divisor}. Then, it suffices to show the following two equations. 
\begin{align*}
b^{(k)}_0&=\frac{1}{(n+1)!}(\mathcal{L}^{(k)})^{n+1},\\
\tilde{b}^{(k)}_{0,1}&=\frac{1}{n!}(\mu_k^*(D\times\mathbb{A}^1))\cdot (\mathcal{L}^{(k)})^n.
\end{align*}
 Since $\mathfrak{a}_{(k)}|_{D\times\mathbb{A}^1}$ is generated by $\sum t^{-\lambda}(\mathscr{F}^\lambda_{(k)})_{D}H^0(D,kL|_D)$ and $\mu_k^*(D\times\mathbb{A}^1)$ is the blow up along $\mathfrak{a}_{(k)}|_{D\times\mathbb{A}^1}$, the latter equation follows from \cite[Proposition 2.6]{M}. The former follows in the same way. Thus, (a) holds. We conclude that (\ref{limjlim}) holds by (a) and (b). 
 \end{proof}

\begin{rem}\label{szerem}
By Lemma \ref{lim}, we can define $(\mathcal{J}^H)^{\mathrm{NA}}(\mathscr{F})$ of a non finitely generated filtration to be $ \lim_{k\to\infty}(\mathcal{J}^H)^{\mathrm{NA}}(\mathscr{F}_{(k)})$. Note that $\mathscr{F}_{(k)}$ is good and we can define $\mathrm{DF}(\mathscr{F}_{(k)})$ for any $k$. Sz\'{e}kelyhidi defined the Futaki invariant of a non finitely generated filtration $\mathscr{F}$ to be $$\mathrm{Fut}(\mathscr{F})=\liminf_{k\to\infty}\mathrm{DF}(\mathscr{F}_{(k)})$$ in \cite{Sz}. There is a subtle but nontrivial problem that if $\mathscr{F}$ is good, we do not know whether $\mathrm{DF}(\mathscr{F})\ge\liminf_{k\to\infty}\mathrm{DF}(\mathscr{F}_{(k)})$ or not in contrast to Lemma \ref{lim}. This problem is closely related to \cite[Conjecture 2.5]{BJ2}. This is why we applied \cite[Proposition 11]{Sz} instead of [{\it loc.cit.}, Theorem 10] to deduce Theorem \ref{ch}.
\end{rem}
 \section{Proof of the main theorems}\label{3}
 \subsection{Construction of a good filtration}
 Before proving our main results, we first define CM degrees.
 \begin{de}\label{notate}
 Let $\pi:(X,L)\to C$ be a projective flat morphism from a normal variety $X$ with a $\mathbb{Q}$-line bundle $L$ to a smooth curve $C$. Fix a closed point $0\in C$ and $\mathrm{dim}\,X=n+1$. Let $C^\circ=C\setminus0$ and assume that $\pi_*\mathcal{O}_{X}\cong\mathcal{O}_{C}$. If $L$ is (semi)ample, then we call $\pi$ a {\it polarized (resp., semiample) family} over a curve $C$.
 
 Furthermore, let $\pi:(X,\Delta,L)\to C$ be a morphism such that $\pi:(X,L)\to C$ is a polarized (resp., semiample) family over $C$ and $\Delta$ is an effective horizontal $\mathbb{Q}$-divisor (i.e., any irreducible component dominates $C$). Then we call this a (log) polarized (resp., semiample) family over $C$. If $K_X+\Delta$ is further $\mathbb{Q}$-Cartier, we call this a $\mathbb{Q}$-Gorenstein family. We denote the fiber over $0$ by $(X_0,\Delta_0,L_0)$. Here, $\Delta_0$ and $L_0$ are the restrictions to $X_0$ as $\mathbb{Q}$-divisors of $\Delta$ and $L$ respectively.
 
 Let $\pi:(X,L)\to C$ and $\pi':(X',L')\to C$ be semiample families. $f:(X,L)\to (X',L')$ is called a {\it $C$-isomorphism} if $f$ is an isomorphism between $X$ and $X'$ preserving the structure morphisms to $C$ such that $L\sim_{\mathbb{Q},C}f^*L'$. Let $\pi:(X,\Delta,L)\to C$ and $\pi':(X',\Delta',L')\to C$ be log semiample families over $C$. We say that $f:(X,\Delta,L)\to (X',\Delta',L')$ is a $C$-isomorphism of log semiample families, if $f$ is a $C$-isomorphism from $(X,L)$ to $(X',L')$ as semiample families such that $f_*\Delta=\Delta'$. We define a $C^{\circ}$-isomorphism of semiample families between $(X,\Delta,L)\times_CC^\circ:=(X\times_CC^\circ,\Delta\times_CC^\circ,L|_{X\times_CC^\circ})$ and $(X',\Delta',L')\times_CC^\circ$ in the same way.

  Suppose that $\pi:(X,L)\to C$ is a semiample family over a proper smooth curve $C$. We define the {\it CM degree}
 \[
 \mathrm{CM}((X,L)/C)=2\frac{b_0a_1-b_1a_0}{a_0^2}+2(1-g(C)),
 \]
 where $g(C)$ is the genus of $C$, $\chi(X,kL)=b_0k^{n+1}+b_1k^n+O(k^{n-1})$ and $\chi(X_0,kL_0)=a_0k^{n}+a_1k^{n-1}+O(k^{n-2})$ for sufficiently divisible $k\in\mathbb{Z}_{>0}$ (cf., \cite{Oh}). This is the positive multiple of the degree of the CM line bundle (cf., \cite{FR}) and we note that if $L\sim_{\mathbb{Q},C}rL'$ then $ \mathrm{CM}((X,L)/C)= \mathrm{CM}((X,L')/C)$ for any $r\in\mathbb{Q}_{>0}$. 
 
 Furthermore, let $\Delta$ be an effective horizontal $\mathbb{Q}$-divisor on $X$. Suppose that $\chi(\Delta,mL)=\tilde{b}_0m^n+O(m^{n-1})$ and $\chi(\Delta_0,mL_0)=\tilde{a}_0m^{n-1}+O(m^{n-2})$ for sufficiently divisible $m\in\mathbb{Z}_{>0}$ (cf., \S\ref{defkst}). Here we abusively denote $L|_{D_i}$ by $L$. Then, we define
 \[
 \mathrm{CM}((X,\Delta,L)/C)=\mathrm{CM}((X,L)/C)-\frac{b_0\tilde{a}_0-\tilde{b}_0a_0}{a_0^2}
 \]
 the {\it log CM degree}. It is well-known that the following holds as \cite{O3} and \cite{W}
 \begin{equation}\label{CMform}
 \mathrm{CM}((X,\Delta,L)/C)=\frac{1}{L_0^n}\left((K_{X/C}+\Delta)\cdot L^n-\frac{n}{n+1}\frac{(K_{X_0}+\Delta_0)\cdot L_0^{n-1}}{L_0^n}L^{n+1}\right).
 \end{equation}
We remark that we can define the intersection number $(K_{X/C}+\Delta)\cdot L^n$ since $K_X+\Delta$ is Cartier in codimension 1 (cf., \cite[Lemma 3.5]{O3}). 
 \end{de}
 As DF invariants, the following hold by the equation (\ref{CMform}) and by the same argument as in \cite[\S7]{BHJ}. The proof is easy and left to the reader.
 \begin{lem}\label{fomulalem}
 Let $\pi:(X,\Delta,L)\to C$ be as above. Then the following hold.
 \begin{enumerate}
 \item For any finite morphism $f:C'\to C$ of smooth curves of degree $r$, let $(X',\Delta',L')$ be the normalization of $(X,\Delta,L)\times_C C'$. Then we have
 \[
  \mathrm{CM}((X',\Delta',L')/C')\le r \,\mathrm{CM}((X,\Delta,L)/C).
 \]
 \item  For any birational morphism $\mu:X'\to X$ from a normal variety that is isomorphic over $C^\circ$, let $\Delta'$ be the strict transformation of $\Delta$. Then
  \[
  \mathrm{CM}((X',\Delta',\mu^*L)/C)= \mathrm{CM}((X,\Delta,L)/C).
 \]
 \end{enumerate}
 \end{lem}

 Odaka proposed the following conjecture.
 
 \begin{conj}[CM minimization, cf., {\cite[Conjecture 8.1]{O5}}]\label{CMconj}
 Let $\pi:(X,\Delta,L)\to C$ be a $\mathbb{Q}$-Gorenstein polarized family such that $(X_0,\Delta_0,L_0)$ is K-semistable. Then 
 \[
 \mathrm{CM}((X,\Delta,L)/C)\le \mathrm{CM}((X',\Delta',L')/C)
 \]
 for any polarized family $\pi':(X',\Delta',L')\to C$ such that there exists a $C^{\circ}$-isomorphism $f^{\circ}:(X,\Delta,L)\times_CC^{\circ}\cong (X',\Delta',L')\times_CC^{\circ}$. Furthermore, if $(X_0,\Delta_0,L_0)$ is K-stable then equality holds iff $f^{\circ}$ extends to $f:(X,\Delta,L)\cong(X',\Delta',L')$ over $C$ entirely.
 \end{conj}
 
 \begin{rem}
 In Conjecture \ref{CMconj}, we assume that $X$ and $X'$ are normal. If we do not assume so, $(X,\Delta,L)$ and $(X',\Delta',L')$ are not isomorphic entirely in general even if $(X_0,\Delta_0,L_0)$ is specially K-stable (cf., Definition \ref{mmt} below). This phenomenon was observed for test configurations by \cite{LX}. Thus, if $(X_0,\Delta_0,L_0)$ is K-stable but $X'$ is not normal, then $(X,\Delta,L)$ and $(X',\Delta',L')$ are conjectured to be isomorphic only in codimension 1.
 \end{rem}
 
 Conjecture \ref{CMconj} is proved in Calabi-Yau case by \cite{O2}, in K-ample case by \cite{WX} and in K-(semi)stable log Fano case by \cite{X}. Note that their results follow from Theorems \ref{spec} and \ref{CM2}.
 
 We prove Conjecture \ref{CMconj} when $(X_0,\Delta_0,L_0)$ is 
 \begin{itemize}
 \item a cscK manifold with the discrete automorphism group in \S \ref{3.2} and
 \item specially K-stable (cf., Definition \ref{mmt}) in \S\ref{mtp2}.
 \end{itemize}
  To do this, we first prove that the difference of CM degrees is the DF invariant of a certain good filtration.

 \begin{thm}\label{CM}
Let $\pi:(X,L)\to C$ and $\pi':(X',L')\to C$ be polarized families over a projective smooth curve such that there exists a $C^\circ$-isomorphism $f^\circ:(X,L)\times_CC^\circ\cong (X',L')\times_CC^\circ$. Then there exist a positive integer $k\in\mathbb{Z}_{>0}$ and a good filtration $\mathscr{F}$ of $(X_0,kL_0)$ such that its weight is 
\[
w_{\mathscr{F}}(m)=\chi(X',mk(L'+aX'_0))-\chi(X,mkL)+O(m^{n-1})
\]
for some $a\in\mathbb{Z}_{>0}$. In particular, 
 \[
 \mathrm{DF}(\mathscr{F})=  \mathrm{CM}((X',L')/C)-\mathrm{CM}((X,L)/C).
 \]
 \end{thm}
 \begin{proof}
We may assume that there exists a morphism $\mu:X'\to X$ such that $\mu|_{X'^\circ}=f^{\circ-1}$ by Lemma \ref{fomulalem} (2). We identify $L'=\mu^*L-E$ where $E$ is an effective divisor whose support is contained in $X'_0$ by replacing $L'$ by $L'+aX'_0$ for some $a\in\mathbb{Z}_{>0}$. Then there exists a canonical injection
\[
H^0(X',mL')\subset H^0(X,mL).
\]
Due to Lemma \ref{fomulalem} (2), we may assume that $-E$ is $\mu$-ample by replacing $X'$ by the ample model $\mathrm{Proj}_X(\bigoplus_{m\ge0} \mu_*\mathcal{O}_{X'}(mL))$ of $(X',L')$ over $X$. 

We may assume that $L$ and $L'$ are $\mathbb{Z}$-Cartier by replacing $L$ by $kL$ for some $k\in\mathbb{Z}_{>0}$. We may also assume that $L$ is ample (resp., $L'$ is semiample) since $\chi(X',mL')-\chi(X,mL)$ does not change when we replace $L$ and $L'$ by $L+cX_0$ and $L'+cX'_0$ for some $c\in\mathbb{Z}_{>0}$ respectively. By the Serre vanishing theorem, $h^i(X,mL)=0$ for $i>0$ and sufficiently large $m$. On the other hand, we prove the following claim.
\begin{claim}\label{cl1}
We have
\begin{align*}
h^i(X',mL')=O(m^{n-1}),\textrm{ for }i>0,\\
\chi(X',mL')=h^0(X',mL')+O(m^{n-1})
\end{align*} for sufficiently large $m$.
\end{claim}
Indeed, take the ample model $(X'_{\mathrm{amp}},L'_{\mathrm{amp}})$ of $(X',L')$ over $C$. Recall that $L_{\mathrm{amp}}'$ is ample. Let $\xi:X'\to X'_{\mathrm{amp}}$ be the canonical morphism. Note that $\xi$ is isomorphic in codimension 1 and \cite[III Theorem 11.1]{Ha} implies that $\mathrm{Supp}\,R^j\xi_*\mathcal{O}_{X'}$ is of codimension $\ge2$ for $j>0$. By the Leray spectral sequence $H^i(X'_{\mathrm{amp}},R^j\xi_*\mathcal{O}_{X'}\otimes L'^{m}_{\mathrm{amp}})\Rightarrow H^{i+j}(X',L'^{m})$ and by the Serre vanishing theorem, we obtain $h^i(X',mL')=O(m^{n-1})$ for $i>0$ and hence $\chi(X',mL')=\chi(X'_{\mathrm{amp}},mL'_{\mathrm{amp}})+O(m^{n-1})$. Note also that $h^0(X',mL')=h^0(X'_{\mathrm{amp}},mL'_{\mathrm{amp}})$. Thus we have Claim \ref{cl1} by the Serre vanishing theorem applied to $L_{\mathrm{amp}}$.

\vspace{0.3cm}

On the other hand, let $\mathfrak{a}$ be an ideal on $X$ such that $\mu$ is the blow up of $\mathfrak{a}$ and there exists an integer $k\in\mathbb{Z}_{>0}$ such that $\mu^{-1}\mathfrak{a}=\mathcal{O}(-kE)$. Indeed, the $\mathcal{O}_X$-algebra $\bigoplus_{l\ge0}\mu_*\mathcal{O}(-lE)$ is finitely generated and hence $\mu_*\mathcal{O}(-kE)$ generates $\bigoplus_{l\ge0}\mu_*\mathcal{O}(-lE)$ for some $k\in\mathbb{Z}_{>0}$. We may assume that $k=1$ by replacing $L$ by $kL$. Then, we obtain the following exact sequence,
\[
0\to H^0(X,mL\otimes\mathfrak{a}^m)\to H^0(X,mL)\to H^0(X,mL\otimes(\mathcal{O}_X/\mathfrak{a}^m))\to H^1(X,mL\otimes\mathfrak{a}^m).
\] 
By \cite[Lemma 5.4.24]{Laz}, $H^i(X',mL')=H^i(X,mL\otimes\mathfrak{a}^m)$ for sufficiently large $m>0$. Since $h^1(X',mL')=O(m^{n-1})$ by Claim \ref{cl1}, we obtain 
\begin{align*}
\chi(X',mL')-\chi(X,mL)&=-\mathrm{dim}\,\mathrm{Im}\,(H^0(X,mL)\to H^0(X,mL\otimes(\mathcal{O}_X/\mathfrak{a}^m)))+O(m^{n-1})\\
&=-\mathrm{dim}\,H^0(X,mL\otimes(\mathcal{O}_X/\mathfrak{a}^m))+O(m^{n-1}).
\end{align*}
 Next, take a so small affine open neighborhood $U$ of $0\in C$ that there exists $t\in H^0(U,\mathcal{O}_U)$ such that $\mathcal{O}_U/t=\mathcal{O}_{C,0}/\mathfrak{m}_0$. Here, $\mathfrak{m}_0$ is the maximal ideal of $\mathcal{O}_{C,0}$. Note that the support of $\mathcal{O}_X/\mathfrak{a}$ is contained in $X_0$. Thus, 
 \begin{align*}
 \mathrm{Im}\,(H^0(X,mL)\to H^0(X,mL\otimes(\mathcal{O}_X/\mathfrak{a}^m)))&\subset \mathrm{Im}\,(H^0(\pi^{-1}(U),mL)\to H^0(X,mL\otimes(\mathcal{O}_X/\mathfrak{a}^m)))\\
 &\subset H^0(X,mL\otimes(\mathcal{O}_X/\mathfrak{a}^m)).
 \end{align*}
 Since $$h^0(X,mL\otimes(\mathcal{O}_X/\mathfrak{a}^m))=\mathrm{dim}\,\mathrm{Im}\,(H^0(X,mL)\to H^0(X,mL\otimes(\mathcal{O}_X/\mathfrak{a}^m)))+O(m^{n-1}),$$ we conclude that 
 \[
 \chi(X',mL')-\chi(X,mL)=-\mathrm{dim}\,\mathrm{Im}\,(H^0(\pi^{-1}(U),mL)\to H^0(X,mL\otimes(\mathcal{O}_X/\mathfrak{a}^m)))+O(m^{n-1}).
 \]
 Note also that $$\mathrm{Im}\,(H^0(\pi^{-1}(U),mL)\to H^0(X,mL\otimes(\mathcal{O}_X/\mathfrak{a}^m)))\cong H^0(\pi^{-1}(U),mL)/H^0((\pi\circ\mu)^{-1}(U),mL').$$
In this situation, we have just proved that $$ \chi(X',mL')-\chi(X,mL)=-\mathrm{dim}\,H^0(\pi^{-1}U,mL)/H^0(\pi'^{-1}U,mL')+O(m^{n-1})$$ for sufficiently large $m$. 
  
  Here, we define $F^{-i}H^0(\pi^{-1}U,mL)$ by 
  \begin{equation}
  t^iF^{-i}H^0(\pi^{-1}U,mL)=H^0(\pi'^{-1}U,mL')\cap t^iH^0(\pi^{-1}U,mL)\label{eqdef}
  \end{equation}
   as submodules of $H^0(X,mL)$ for $i\in\mathbb{Z}_{\ge0}$. Let $F^iH^0(\pi^{-1}U,mL)=0$ for $i>0$. Then, $F^{\bullet}H^0(\pi^{-1}U,mL)$ defines a linearly bounded multiplicative filtration of $H^0(\pi^{-1}U,mL)$. Indeed, if we take $i>0$ such that $iX_0'-E$ is effective then for $m>0$, $$t^{mi}H^0(\pi^{-1}U,mL)=H^0(\pi'^{-1}U,mL'+mE-miX'_0)\subset H^0(\pi'^{-1}U,mL').$$ On the other hand, let $\mathscr{F}^iH^0(X_0,mL_0)$ be the image of $F^iH^0(\pi^{-1}U,mL)\to H^0(X_0,mL_0)$. It is easy to see that $\mathscr{F}^{\bullet}H^0(X_0,mL_0)$ is a linearly bounded multiplicative filtration. Let $w_{\mathscr{F}}(m)$ be the weight of $\mathscr{F}^{\bullet}H^0(X_0,mL_0)$. Then it suffices to show that $$w_{\mathscr{F}}(m)=-\mathrm{dim}\,H^0(\pi^{-1}U,mL)/H^0(\pi'^{-1}U,mL')=\chi(X',mL')-\chi(X,mL)+O(m^{n-1}).$$ This equation follows from Claim \ref{cl2} below.
   \begin{claim}\label{cl2}
For $m>0$ such that $H^1(X_0,mL_0)=0$, there exists an isomorphism
\begin{align*}
H^0(\pi^{-1}U,mL)/H^0(\pi'^{-1}U,mL')\cong \bigoplus_{i\ge0} t^{i} (H^0(X_0,mL_0)/\mathscr{F}^{-i}H^0(X_0,mL_0))
\end{align*}
as $k$-vector spaces.
\end{claim} For this, note that $H^0(\pi^{-1}U,mL)\to H^0(X_0,mL_0)$ is surjective and
\[
H^0(\pi'^{-1}U,mL')= \sum_{i\ge0} t^iF^{-i}H^0(\pi^{-1}U,mL)\subset H^0(\pi^{-1}U,mL).
\]
Note also that there is a short exact sequence $$0\to \mathcal{V}'/\mathcal{V}'\cap t\mathcal{V}\to \mathcal{V}/t\mathcal{V}\to(\mathcal{V}/t\mathcal{V})/(\mathrm{Image}(\mathcal{V}'\to\mathcal{V}/t\mathcal{V}))\to0$$ if $\mathcal{V}'\subset\mathcal{V}$ is an inclusion of coherent $\mathcal{O}_{U}$-modules. Then there exists the following commutative diagram whose rows and columns are exact.
 \[
 \begin{CD}
 @.0@.0@.0@.\\
 @.@VVV @VVV @VVV@. \\
0@>>>\mathcal{V}'\cap t\mathcal{V}@>>> t\mathcal{V}@>>>t\mathcal{V}/\mathcal{V}'\cap t\mathcal{V}@>>>0 \\
@.@VVV @VVV @VVV@. \\
0@>>>\mathcal{V}' @>>>  \mathcal{V}@>>>\mathcal{V}/\mathcal{V}'@>>>0\\
@.@VVV @VVV @VVV@. \\
0@>>> \mathcal{V}'/\mathcal{V}'\cap t\mathcal{V}@>>> \mathcal{V}/t\mathcal{V}@>>>(\mathcal{V}/t\mathcal{V})/(\mathrm{Image}(\mathcal{V}'\to\mathcal{V}/t\mathcal{V}))@>>>0\\
@.@VVV @VVV @VVV@. \\
@.0@.0@.0@.
\end{CD}
\]
By the above diagram and by the equation (\ref{eqdef}) applied to the case when $\mathcal{V}=H^0(\pi^{-1}U,mL)$ and $\mathcal{V}'=H^0(\pi'^{-1}U,mL')$, we have the following isomorphism
\begin{align*}
H^0(\pi^{-1}U,mL)/H^0(\pi'^{-1}U,mL')&\cong H^0(X_0,mL_0)/\mathscr{F}^0H^0(X_0,mL_0)\\
&\oplus tH^0(\pi^{-1}U,mL)/tF^{-1}H^0(\pi^{-1}U,mL)
\end{align*}
 as $k$-vector spaces. We apply the induction hypothesis to $\sum_{i\ge1}t^{i}F^{-i}H^0(\pi^{-1}U,mL)$ a submodule of $tH^0(\pi^{-1}U,mL)$ and have $$tH^0(\pi^{-1}U,mL)/tF^{-1}H^0(\pi^{-1}U,mL)\cong\bigoplus_{i\ge1} t^{i} (H^0(X_0,mL_0)/\mathscr{F}^{-i}H^0(X_0,mL_0)).$$ Hence we obtain $$H^0(\pi^{-1}U,mL)/H^0(\pi'^{-1}U,mL')\cong\bigoplus_{i\ge0} t^{i} (H^0(X_0,mL_0)/\mathscr{F}^{-i}H^0(X_0,mL_0)).$$

 Thus, we finish the proof of Claim \ref{cl2} and obtain $w_{\mathscr{F}}(m)=\chi(X',mL')-\chi(X,mL)+O(m^{n-1})$ by \cite[Lemma 2.14]{M}. On the other hand, if $\chi(X',mL')=b'_0m^{n+1}+b'_1m^n+O(m^{n-1})$, then $w_{\mathscr{F}}(m)=(b'_0-b_0)m^{n+1}+(b'_1-b_1)m^n+O(m^{n-1})$. Therefore, $$\mathrm{DF}(\mathscr{F})=\mathrm{CM}((X',L')/C)-\mathrm{CM}((X,L)/C).$$
 We finish the proof.
 \end{proof}
 
\begin{rem}\label{remark.prf}
If $X_0$ is irreducible, $\hat{X}_0$ is the strict transformation of $X_0$ in $X'$, and $E$ is $\mu$-exceptional, then $$\mathscr{F}^{-i}H^0(X_0,mL_0)=\mathrm{Im}\,(H^0(\pi'^{-1}U,mL'+i(X'_0-\hat{X}_0))\to H^0(X_0,mL_0))$$ for $i\in\mathbb{Z}_{\ge0}$ in the above proof. If $(X,L)$ and $(X',L')$ are $\mathbb{Q}$-Fano families with K-semistable fibers, then $\mathscr{F}$ coincides with the one constructed in \cite[\S5]{BX}.
\end{rem}

 We obtain the log version of Theorem \ref{CM} as follows. 
 \begin{cor}\label{logCm}
 Let $f:(X,\Delta,L)\to C$ and $f':(X',\Delta',L')\to C$ be polarized log 
 families over a proper smooth curve $C$. Suppose that $(X_0,\Delta_0)$ is deminormal 
 and there exists a $C^\circ$-isomorphism $(X,\Delta,L)\times_{C} C^\circ\to(X',\Delta',L')\times_{C} C^\circ$. 
  
 Then there exist $k\in\mathbb{Z}_{>0}$ and a good filtration $\mathscr{F}$ of $\bigoplus_{m\ge0}H^0(X_0,mkL_0)$ such that 
 \[
 \mathrm{DF}_{\Delta_0}(\mathscr{F})\le\mathrm{CM}((X',\Delta',L')/C)-\mathrm{CM}((X,\Delta,L)/C).
 \]
 \end{cor}
 
 \begin{proof}
 As the proof of Theorem \ref{CM}, we may replace $(X',\Delta',L')$ by semiample one and may assume that there exists a birational contraction $\mu:X'\to X$ and $L'=\mu^*L-E$ where $E$ is an effective divisor such that $-E$ is $\mu$-ample. Then we take $\mathscr{F}$ as a filtration appears in the proof of Theorem \ref{CM}. We may assume that $\mathscr{F}$ is a filtration of $\bigoplus_{m\ge0} H^0(X_0,mL_0)$ by replacing $L$ by $kL$. It is easy to see that the assertion follows from Lemma \ref{cl3} below applied to an irreducible component $D$ of $\Delta$.
 \end{proof}
  \begin{lem}\label{cl3}
  Let $D$ be a horizontal prime divisor on $X$ such that $D_0=\sum m_i \Gamma_i$ is the irreducible decomposition and $D'\subset X'$ be the strict transformation of $D$. Then $$\frac{1}{n!}(D'\cdot L'^{n}-D\cdot L^n)\ge \sum m_iB_{\mathscr{F}_{1,i}},$$ where $\mathscr{F}_{1,i}$ is the induced filtration on $\bigoplus H^0(\Gamma_i,mL|_{\Gamma_i})$ by $\mathscr{F}$ with the weight function $w_{\mathscr{F}_{1,i}}(m)$ 
  such that $\lim_{m\to\infty}\frac{w_{\mathscr{F}_{1,i}}(m)}{m^n}=B_{\mathscr{F}_{1,i}}$.
  \end{lem}
 \begin{proof}
Let $U$ be an affine open neighborhood of $0\in C$ and $\mathfrak{a}\subset\mathcal{O}_X$ be an ideal as in proof of Theorem \ref{CM}. Instead of Claim \ref{cl1}, we see that 
\begin{align*}
\chi(D',mL')&=h^0(D',mL')+O(m^{n-1})\\
\chi(D,mL)&=h^0(D,mL)+O(m^{n-1})=h^0(D',m\mu^*L)+O(m^{n-1})
\end{align*}
 by \cite[Lemma 2.7]{M}. Note that $D'$ is the blow up of $D$ along $\mathfrak{a}|_{D}$. Hence, $H^i(D',mL')=H^i(D,mL\otimes\mathfrak{a}|_D^m)$ for sufficiently large $m>0$ by \cite[Lemma 5.4.24]{Laz}. Thus, $$\chi(D',mL')-\chi(D,mL)=-\mathrm{dim}\,(H^0(D|_{U},mL)/H^0(D|_{U},mL\otimes\mathfrak{a}^m|_{D|_{U}}))+O(m^{n-1})$$ holds as the proof of Theorem \ref{CM}. Here, $D|_{U}=\pi^{-1}(U)\cap D$. 

 Define a filtration $G^{\bullet}$ on $\bigoplus H^0(D|_U,mL|_{D|_U})$ by $$t^iG^{-i}H^0(D|_U,mL)=H^0(D'|_U,mL')\cap t^iH^0(D|_U,mL)$$ if $i\ge0$. Otherwise, $G^{-i}H^0(D|_U,mL)=0$. Then, let $\mathscr{G}^iH^0(D^*_0,mL_0)$ be the image of $G^iH^0(D|_U,mL)\to H^0(D^*_0,mL_0)$. Here, $D^*_0$ is the scheme theoretic fiber of $D$ over $0$. Let $\mathscr{F}_{1,*}$ be the induced filtration of $\bigoplus H^0(D^*_0,mL|_{D^*_0})$ by $\mathscr{F}$
 . By Proposition \ref{restr}, 
 \begin{equation}
  B_{\mathscr{F}_{1,*}}=\limsup_{m\to\infty}\frac{w_{\mathscr{F}_{1,*}}(m) }{m^n}\ge \sum m_iB_{\mathscr{F}_{1,i}}.\label{equa1}
 \end{equation}
 It is easy to see that $\mathscr{F}_{1,*}^iH^0(D^*_0,mL_0)\subset\mathscr{G}^iH^0(D^*_0,mL_0)$. Therefore, 
  \begin{equation}
B_{\mathscr{G}}=\limsup_{m\to\infty}\frac{ w_{\mathscr{G}}(m)}{m^n}\ge B_{\mathscr{F}_{1,*}}\label{equa2}.
 \end{equation}
On the other hand, we see that 
   \begin{equation}
 w_{\mathscr{G}}(m)=-\mathrm{dim}\,(H^0(D|_{U},mL)/H^0(D|_{U},mL\otimes\mathfrak{a}^m|_{D|_{U}}))=\chi(D',mL')-\chi(D,mL)+O(m^{n-1})\label{equa3}
 \end{equation}
 by Claim \ref{cl2}. Thus, we obtain the assertion by (\ref{equa1}), (\ref{equa2}) and (\ref{equa3}).
 \end{proof}
 

\subsection{Minimizing CM degree for cscK manifolds with the discrete automorphism group}\label{3.2}

 Due to Theorem \ref{CM}, we can prove Conjecture \ref{CMconj} for certain cscK manifolds. The following is Theorem \ref{mt1}.
 
 \begin{cor}\label{cmcsck}
 Let $\pi:(X,L)\to C$ and $\pi':(X',L')\to C$ be polarized families over a proper smooth curve $C$ such that there exists a $C^\circ$-isomorphism $f^\circ:(X,L)\times_CC^\circ\cong (X',L')\times_CC^\circ$. If $X_0$ is smooth and $(X_0,L_0)$ has the discrete automorphism group and a cscK metric, then 
 \[
 \mathrm{CM}((X',L')/C)\ge\mathrm{CM}((X,L)/C).
 \]
Equality holds iff the birational map $f^\circ$ can be extended to a $C$-isomorphism $(X,L)\cong (X',L')$.
 \end{cor}
 \begin{proof}
We use the notations in the proof of Theorem \ref{CM}. That is, we assume that there exists a birational morphism $\mu:X'\to X$ and $L'$ is semiample. By Theorem \ref{prop11}, to see the last assertion, it suffices to show that $\| \mathscr{F}\|_2>0$ when we assume that $\mu^*L\not\sim_{C,\mathbb{Q}}L'$, i.e., the ample model of $(X',L')$ is not isomorphic to $(X,L)$. Let $G$ be the concave transformation (\S\ref{vol}) with respect to $\mathscr{F}$. On $\Delta^\circ$, which is the interior of the Okounkov body $\Delta$ of $(X_0,L_0)$, note that $G$ is continuous. Then it suffices to show $G|_{\Delta^\circ}$ is not a constant function by \cite[Lemma 7]{Sz}. Assume the contrary and we deduce a contradiction.

 We may assume that $E\ne0$ and is $\mu$-exceptional since $X_0$ is irreducible. We may further assume that $-E$ is $\mu$-ample as Theorem \ref{CM}. Let $L_s=\mu^*L-sE$ for $0\le s\le1$ and $f(s)=L_s^{n+1}-L^{n+1}$. Then $f(0)=0$ and
 \[
 \frac{d}{ds}f(s)=-(n+1)E\cdot L_s^{n}\le0.
 \] $\frac{d}{ds}f(s)<0$ for $0<s<1$ since $L_s$ is ample and $E\ne0$ is effective. Hence, 
 \begin{equation}
f(1)=L'^{n+1}-L^{n+1}<0.\label{eqlemma}
 \end{equation} We claim that then $G|_{\Delta^\circ}\equiv0$. Indeed, 
 \[
 \mathscr{F}^0H^0(X_0,mL_0)=\mathrm{Im}\,(H^0(\pi^{-1}U',mL')\to H^0(X_0,mL_0))
 \]
  holds now. This contains the image of $H^0(X',mL')\to H^0(\hat{X}_0,mL'|_{\hat{X}_0})$ and the kernel of this map is $H^0(X',mL'-\hat{X}_0)$ where $\hat{X}_0$ is the strict transformation of $X_0$. Since $L'$ is semiample, $H^0(X',mL'-\hat{X}_0)\ne H^0(X',mL')$ for sufficiently large $m$ and hence there exists at least one rational point $p\in\Delta$ such that $G(p)=0$. Note that $G\le0$. By the concavity of $G$, for any small $\epsilon>0$, there exists a point $q\in\Delta^\circ$ such that $G(q)>-\epsilon$. 
  Since we assume that $G|_{\Delta^\circ}$ is constant, $G|_{\Delta^\circ}\equiv0$. From Theorem \ref{CM}, it follows that $L'^{n+1}=L^{n+1}$. This contradicts to the inequality (\ref{eqlemma}).
 \end{proof}

 \subsection{CM minimization for specially K-stable varieties}\label{mtp2}
 We introduce special K-stability as follows.
  \begin{de}[Specially K-stable varieties]\label{mmt}
 Let $(X,\Delta,L)$ be a polarized slc pair. If there exists $\epsilon>0$ such that $\delta(X,\Delta,L)-\epsilon\in\mathbb{Q}$ and $(X,\Delta,L)$ is uniformly $\mathrm{J}^{K_{X}+\Delta+(\delta(X,\Delta,L)-\epsilon)L}$-stable and $K_{X}+\Delta+\delta(X,\Delta,L)L$ is ample, we say $(X,\Delta,L)$ is {\it specially} K-{\it stable}.
 
 Similarly, if $(X,\Delta,L)$ is $\mathrm{J}^{K_{X}+\Delta+(\delta(X,\Delta,L)+\epsilon)L}$-semistable for any $\epsilon\ge0$ such that $\delta(X,\Delta,L)+\epsilon\in\mathbb{Q}$ and $K_{X}+\Delta+\delta(X,\Delta,L)L$ is nef, then we say $(X,\Delta,L)$ is {\it specially} K-{\it semistable}.
 \end{de}
 \begin{rem}
We see that special K-stability implies uniform K-stability in Corollary \ref{corf1} below.
 
 On the other hand, K-stable varieties are not specially K-stable in general. Indeed, let $(X,-\epsilon K_X+L)$ be a polarized ruled surface where $L$ is the fiber class for $\epsilon>0$. It is well-known that $(X,-\epsilon K_X+L)$ is K-stable when a corresponding vector bundle is stable by \cite{A}. Furthermore, $\lim_{\epsilon\to0}\delta(X,-\epsilon K_X+L)=2$ by \cite[Theorem D]{Hat}. However, $K_X+\delta(X,-\epsilon K_X+L)(-\epsilon K_X+L)$ is not nef for sufficiently small $\epsilon>0$. 
 \end{rem}

The following are known to be specially K-stable.
 \begin{thm}\label{spec}
 Suppose that a polarized lc pair $(X,\Delta,L)$ satisfies one of the following.
 \begin{enumerate}
 \item (K-ample, Calabi-Yau and log Fano cases, \cite{O}, \cite{BlJ}, \cite{LXZ}) There exists a constant $c\in\mathbb{Q}$ such that $K_X+\Delta\sim_{\mathbb{Q}} cL$ and $(X,\Delta,L)$ is K-stable.
 \item (Klt minimal models, \cite{Hat2}) $(X,\Delta)$ is klt, $K_X+\Delta$ is nef and $L=K_X+\Delta+\epsilon H$ for $H$ is ample and sufficiently small $\epsilon>0$.
  \item (K-ample fibrations over curves, \cite{Hat2}) There exists a morphism $f:X\to C$ such that $(C,A)$ is a polarized smooth curve, $f_*\mathcal{O}_X\cong\mathcal{O}_C$, $K_{X}+\Delta$ is ample, and $L=\epsilon (K_X+\Delta)+f^*A$ for sufficiently small $\epsilon>0$.
 \item (Uniformly adiabatically K-stable klt-trivial fibrations over curves, \cite{Hat}) There exists a polarized klt-trivial fibration $f:(X,\Delta,H)\to (C,A)$ such that $(C,A)$ is a polarized smooth curve, $(X,\Delta)$ is klt, $(C,B,M,A)$ is K-stable and $L=f^*A+\epsilon H$ for $H$ and for sufficiently small $\epsilon>0$. Here, $B$ is the discriminant and $M$ is the moduli divisor.
 \end{enumerate}
 Then $(X,\Delta,L)$ is specially K-stable.
 \end{thm}  
 
 It is easy to see that if there exists a constant $c\in\mathbb{Q}$ such that $K_X+\Delta\equiv cL$ and $(X,\Delta,L)$ is K-semistable, then $(X,\Delta,L)$ is specially K-semistable. On the other hand, there exists a polarized lc minimal model $(X,\Delta,H)$ such that $(X,\Delta,\epsilon H+K_X+\Delta)$ is K-unstable for sufficiently small $\epsilon>0$ by \cite[Remark 3.5]{Hat}.

In \cite[Appendix]{Hat}, we see that the sum of the nA J-functional and the log-twisted Ding invariant is a lower bound of the DF invariant of a normal semiample test configuration. In light of Theorem \ref{CM}, it is quite natural to consider to give a lower bound of a CM degree in a similar way. With this in mind, we define the J-degree and the log-twisted Ding degree for any polarized family and see their properties. 

\subsubsection{Minimization for J-degree}\label{1st}First, we introduce the following generalization of nA J-functionals.
 \begin{de}\label{relcomp}
 Let $\pi:(X,L)\to C$ be a polarized family over a proper smooth curve and $\pi':(X',L')\to C$ be another semiample family such that $(X,L)\times_CC^{\circ}\cong (X',L')\times_CC^{\circ}$. Suppose that $H$ is a $\mathbb{Q}$-line bundle on $X$ and there exists a canonical birational morphism $\mu:X'\to X$. Then, we define the J$^H$-degree of $(X',L')$ as
 \[
 \mathcal{J}^H(X',L')=\frac{1}{L_0^n}\left(\mu^*H\cdot L'^n-\frac{n}{n+1}\frac{H_0\cdot L_0^{n-1}}{L_0^n}L'^{n+1}\right).
 \]
 Note that this degree is pullback invariant (i.e., Lemma \ref{fomulalem} holds for J$^H$-degrees) and for general $\mathbb{Q}$-divisor $D\sim_{\mathbb{Q}}H$ such that the support of $D'=\mu^*D$ contains no component of $X_0'$, we deduce from \cite[Corollary 1.4.41]{Laz} that
 \[
  \mathcal{J}^H(X',L')=\frac{1}{a_0^2}\left(a_0\left(\lim_{m\to\infty}\frac{\chi(D',mL')}{m^{n}}\right)-b'_0\left(\lim_{m\to\infty}\frac{\chi(D_0,mL_0)}{m^{n-1}}\right)\right).
 \]
Here, $a_0=\lim_{m\to\infty}\frac{\chi(X_0,mL_0)}{m^{n}}$ and $b_0'=\lim_{m\to\infty}\frac{\chi(X',mL')}{m^{n+1}}$. If this is the case, we call $D$ a {\it compatible} divisor with $\mu$. Here, we remark that a compatible divisor $D$ is not necessarily effective.
 \end{de}

Note also that $ \mathcal{J}^H(X',L')$ is independent from relative linear equivalence class of $L'$ over $C$. The following is well-known.

\begin{lem}\label{wx}
Notations as above. If $H\equiv L$, then
 \[
 \mathcal{J}^H(X',L')\ge  \mathcal{J}^H(X,L).
 \]
Equality holds iff $\mu^*L\sim_{\mathbb{Q},C}L'$.
\end{lem}

\begin{proof}
Let $L'=\mu^*L+E$ where the support of $E$ is contained in $X'_0$. Let $L'_s=\mu^*L+sE$ and $f(s)=\mathcal{J}^H(X',L'_s)-\mathcal{J}^H(X,L)$ for $s\in[0,1]$. Then by \cite[Lemma 1]{LX}
\[
\frac{d}{ds}f(s)=-\frac{ns}{(L_0)^n}(E^2)\cdot(L_s'^{n-1})\ge0.
\]
Moreover, this derivative is positive when $E\not\sim_{C,\mathbb{Q}}0$ and $s\in(0,1)$.
\end{proof}

 By Lemma \ref{lim}, if $(X,L)$ is J$^H$-semistable then $(\mathcal{J}^H)^{\mathrm{NA}}(\mathscr{F})\ge0$ for any filtrations. With this in mind, we prove the following, so-called J-minimization.
\begin{prop}\label{jcm2}
Let $\pi:(X,L)\to C$ be a polarized family over a proper smooth curve and $\pi':(X',L')\to C$ be another semiample family such that there exists a $C^\circ$-isomorphism $f^\circ:(X,L)\times_CC^{\circ}\cong (X',L')\times_CC^{\circ}$. Let $H$ be a $\mathbb{Q}$-line bundle on $X$ such that $H_0$ is nef. Suppose that $(X_0,L_0)$ is a polarized $\mathrm{J}^{H_0}$-semistable deminormal scheme and there exists a birational morphism $\mu:X'\to X$ over $C$ such that $\mu|_{X'^\circ}=f^{\circ-1}$. 

Then the following inequality holds.
\begin{equation}\label{J-eq}
\mathcal{J}^H(X',L')\ge\mathcal{J}^H(X,L).
\end{equation}
Furthermore, if $H_0$ is ample and $(X_0,L_0)$ is uniformly {\rm J}$^{H_0}$-stable, then equality holds iff $\mu^*L\sim_{\mathbb{Q},C}L'$.
\end{prop}

\begin{proof}
First, we treat the case when $H_0$ is ample and $(X_0,L_0)$ is uniformly {\rm J}$^{H_0}$-stable and show (\ref{J-eq}). By Definition \ref{relcomp}, we may replace $\mu$ by the ample model of $(X',L')$ over $X$. Note that $\mu^*D=\mu_*^{-1}D$ holds now. As Theorem \ref{CM}, replacing $L$ by $kL$ for sufficiently divisible $k\in\mathbb{Z}_{>0}$, take a good filtration $\mathscr{F}$ of $(X_0,L_0)$. Now, we may assume that $(X,L)$ and $\mathscr{F}$ satisfy Condition \ref{crucial} and take its approximation $\{\mathscr{F}_{(k)}\}_{k\in\mathbb{Z}_{>0}}$ as Definition \ref{appro}. 
Replacing $H$ by $cH$ for some $c\in\mathbb{Z}_{>0}$, we may assume that $H$ is $\mathbb{Z}$-Cartier, $H_0$ is very ample and there exists an open neighborhood $0\in U\subset C$ such that $H$ is ample over $U$ by \cite[Proposition 1.41]{KM}. Since we work over $\mathbb{C}$, we pick a very general divisor $D\sim H$ up satisfying the following.
\begin{itemize}
\item $D\cap\pi^{-1}U$ is effective, reduced and irreducible, 
\item $D$ has connected fibers and is flat over $U$,
\item $D$ is compatible with $\mu$,
\item The restriction $D_0\in|H_0|$ of $D$ to $X_0$ is compatible with $\{\mathscr{F}_{(k)}\}_{k\in\mathbb{Z}_{>0}}$ (see Definition \ref{compatible.divisor}).
\end{itemize} 
 Let $b_0=\lim_{m\to\infty}\frac{\chi(X,mL)}{m^{n+1}}$. Then we have
 \begin{align*}
  \mathcal{J}^H(X',L')-\mathcal{J}^H(X,L)&=\frac{-1}{a_0^2}\Biggl((b'_0-b_0)\left(\lim_{m\to\infty}\frac{\chi(D_0,mL_0)}{m^{n-1}}\right)\\
  &-a_0\left(\lim_{m\to\infty}\left(\frac{\chi(D',mL')}{m^{n}}-\frac{\chi(D,mL)}{m^{n}}\right)\right)\Biggr),
 \end{align*}
In the above equation, the first term of the right hand side is calculated in Theorem \ref{CM}. To calculate the second term, let $D_1$ be the Zariski closure of $D\cap\pi^{-1}U$ in $X$. 
Then 
\[
\left(\frac{\chi(\mu^*D,mL')}{m^{n}}-\frac{\chi(D,mL)}{m^{n}}\right)=\left(\frac{\chi(\mu^{-1}_*D_1,mL')}{m^{n}}-\frac{\chi(D_1,mL)}{m^{n}}\right)
\]
holds and we apply Lemma \ref{cl3} to an effective divisor $D_1$ instead of $D$. Then, we obtain the inequality
\[
\mathcal{J}^H(X',L')-\mathcal{J}^H(X,L)\ge(\mathcal{J}^H)^\mathrm{NA}(\mathscr{F}).
\]
By 
Lemma \ref{lim} and the assumption that $(X_0,L_0)$ is J$^{H_0}$-semistable, we have $(\mathcal{J}^H)^\mathrm{NA}(\mathscr{F})\ge0$.

On the other hand, $(X_0,L_0)$ is uniformly $\mathrm{J}^{(H_0-\epsilon L_0)}$-stable and $H_0-\epsilon L_0$ is ample for sufficiently small $\epsilon>0$. Therefore, if equality holds in (\ref{J-eq}), so does $\mathcal{J}^L(X',L')=\mathcal{J}^L(X,L)$. Then it follows that $\mu^*L\sim_{\mathbb{Q},C}L'$ from Lemma \ref{wx}.

Finally, suppose that $H_0$ is nef and $(X_0,L_0)$ is J$^{H_0}$-semistable. For $\epsilon>0$, $H_0+\epsilon L_0$ is ample and $(X_0,L_0)$ is uniformly J$^{(H_0+\epsilon L_0)}$-stable. We have already shown that $\mathcal{J}^{H+\epsilon L}(X',L')\ge\mathcal{J}^{H+\epsilon L}(X,L)$. Thus,
\[
 \mathcal{J}^H(X',L')-\mathcal{J}^H(X,L)=\lim_{\epsilon\to 0} (\mathcal{J}^{H+\epsilon L}(X',L')-\mathcal{J}^{H+\epsilon L}(X,L))\ge0.
 \]
 We complete the proof.
\end{proof}
\begin{rem}\label{forapp}
If the base field $k$ is a countable algebraically closed field, Proposition \ref{jcm2} also holds. As in the proof, we take an approximation $\{\mathscr{F}_{(r)}\}_{r\in\mathbb{Z}_{>0}}$ to $\mathscr{F}$ and uncountable algebraically closed field $k'$ containing $k$. Note that in this case, there exists no compatible divisor with $\{\mathscr{F}_{(r)}\}_{r\in\mathbb{Z}_{>0}}$. However, if we change the base field to $k'$, then there exists a compatible divisor $D$. Denote $\mathscr{F}'$ be the filtration of $R\otimes_kk'$ defined by $\mathscr{F}^\bullet R\otimes_kk'$ for example. Let $h:X\times_kk'\to X$ be the canonical morphism. Then it is easy to see that $(\mathcal{J}^{H_0})^{\mathrm{NA}}(\mathscr{F}_{(r)})=(\mathcal{J}^{h_0^*H_0})^{\mathrm{NA}}((\mathscr{F}_{(r)})')$ and $\mathcal{J}^H(X,L)=\mathcal{J}^{h^*H}(X\times_kk',h^*L)$. The same equation holds for $(X',L')$. Thus we have similarly
\[
\mathcal{J}^H(X',L')-\mathcal{J}^H(X,L)\ge\lim_{r\to\infty}(\mathcal{J}^H)^\mathrm{NA}(\mathscr{F}_{(r)}).
\]
\end{rem}
\subsubsection{Minimization for Ding degree}\label{2nd}On the other hand, we consider the following generalization of the log-twisted Ding invariant of a semiample test configuration.
 \begin{de}
 Let $\pi:(X,\Delta,L)\to C$ be a $\mathbb{Q}$-Gorenstein polarized family over a proper smooth curve where $X_0$ is irreducible. Suppose that there exists a $\mathbb{Q}$-line bundle $H$ on $X$ such that $L|_{X\times_{C}C^{\circ}}\sim_{\mathbb{Q}, C^\circ}-(K_X+\Delta+H)|_{X\times_{C}C^{\circ}}$. Let $\pi':(X',\Delta',L')\to C$ be another semiample family such that $(X,\Delta,L)\times_CC^{\circ}\cong (X',\Delta',L')\times_CC^{\circ}$. Suppose that there exists a birational $C$-morphism $\mu:X'\to X$ that induces the $C^{\circ}$-isomorphism. Then we define the {\it log-twisted Ding degree} of $(X',\Delta',L')$ with respect to $H$
 \begin{equation*}\label{dingeqdef}
\begin{split}
 \mathrm{Ding}_{(\Delta',H)}(X',L')&=\sup\{t\in\mathbb{Q}|(X',\Delta'+D_{(X',\Delta',H,L')}+tX_0')\textrm{ is sublc around }X_0'\}\\
 &-1-\frac{L'^{n+1}}{(n+1)L_0^n},
 \end{split}
 \end{equation*}
 where $D_{(X',\Delta',H,L')}$ is a unique $\mathbb{Q}$-divisor whose support is contained in $X'_0$ such that $D_{(X',\Delta',H,L')}\equiv-(K_{X'/C}+\Delta'+L'+\mu^*H)$ holds. Indeed, uniqueness of $D_{(X',\Delta',H,L')}$ follows from the fact that there exists a very general curve $C'\subset X'$ that is disjoint from any irreducible component of $X'_0$ but the strict transformation of $X_0$ and from \cite[Lemma 3.39]{KM}. Note that $K_{X'}+\Delta'+D_{(X',\Delta',H,L')}$ is $\mathbb{Q}$-Cartier and we can check whether $(X',\Delta'+D_{(X',\Delta',H,L')}+tX_0')$ is sublc around $X_0'$ or not. We denote the first term of the right hand side of 
 the above equation by $\mathrm{lct}(X',\Delta'+D_{(X',\Delta',H,L')};X_0')$. It is easy to see that if $L''\sim_{\mathbb{Q},C}L'$, then 
  \[
 \mathrm{Ding}_{(\Delta',H)}(X',L')= \mathrm{Ding}_{(\Delta',H)}(X',L'').
 \]
 \end{de}
In the above definition, if $H=0$ and $\Delta'=0$, then the log-twisted Ding degree coincides with the degree of the Ding line bundle introduced by \cite[(3.5)]{Ber}. The log-twisted Ding degree has similar properties to those of the log-twisted Ding invariant.
\begin{lem}\label{ding.lemma}
Notations as above. Then the following hold.
\begin{enumerate}
\item Let $\pi'':(X'',\Delta'',L'')\to C$ be another semiample $\mathbb{Q}$-Gorenstein family over $C$ such that there exists a birational morphism $\mu'':X''\to X'$ such that $\mu''$ induces a $C^\circ$-isomorphism $(X',\Delta')\times_{C}C^{\circ}\cong(X'',\Delta'')\times_{C}C^{\circ}$ and $L''=\mu''^*L'$. Then 
\[
 \mathrm{Ding}_{(\Delta'',H)}(X'',L'')= \mathrm{Ding}_{(\Delta',H)}(X',L').
\]
\item Let $f:C'\to C$ be a finite morphism of smooth curves and $\pi'':(X'',\Delta'',L'')\to C'$ be the normalized base change of $\pi':(X',\Delta',L')\to C$. Let $\mu'':X''\to X'$ be the induced morphism and $r=\mathrm{deg}\,f$. Then
\[
\mathrm{Ding}_{(\Delta'',H)}(X'',L'')= r\,\mathrm{Ding}_{(\Delta',H)}(X',L').
\]
\item We have the following inequality
\[
\mathrm{Ding}_{(\Delta',H)}(X',L')\le\mathrm{CM}_{(\Delta',H)}(X',L')\coloneq \mathrm{CM}((X',\Delta',L')/C)+ \mathcal{J}^H(X',L').
\]
Equality holds if $(X',\Delta'+X'_0)$ is lc and $K_{X'/C}+\Delta'+L'+\mu^*H\sim_{\mathbb{Q},C}0$.
\end{enumerate}
\end{lem}

 The proof is similar to \cite[Proposition A.12]{Hat} and left to the reader. Here, we prove the following generalization of \cite[Theorem 3.1]{BX} to the log twisted Fano case.
 
 \begin{prop}\label{ding}
 Let $\pi:(X,\Delta,L)\to C$ be a polarized $\mathbb{Q}$-Gorenstein family over a proper smooth curve. Let $H$ be a $\mathbb{Q}$-line bundle such that $H\sim_{\mathbb{Q},C}-(K_{X/C}+\Delta+L)$. Suppose that $\delta(X_0,\Delta_0,L_0)\ge1$. Let $\pi':(X',\Delta',L')\to C$ be a semiample family such that there exists a birational morphism $\mu:X'\to X$ that induces a $C^\circ$-isomorphism $(X,\Delta,L)\times_{C}C^{\circ}\cong(X',\Delta',L')\times_{C}C^{\circ}$. Then
 \[
 \mathrm{Ding}_{(\Delta',H)}(X',L')\ge \mathrm{Ding}_{(\Delta,H)}(X,L).
\]
 \end{prop}
 
 \begin{proof}
First, we consider the case when $\delta(X_0,\Delta_0,L_0)>1$. We may assume that $L'=\mu^*L-E$ where $E$ is an effective exceptional $\mathbb{Q}$-divisor by Lemma \ref{ding.lemma} (1). Take sufficiently divisible $r_0\in\mathbb{Z}_{>0}$ that $r_0L$ and $r_0L'$ are $\mathbb{Z}$-Cartier and that there exists an ideal $\mathfrak{a}\subset \mathcal{O}_X$ such that $\mu^{-1}\mathfrak{a}=\mathcal{O}(-r_0E)$. We may also assume that $-E$ is $\mu$-ample and $X'$ is obtained by blowing $X$ up along $\mathfrak{a}$. As in the proof of Theorem \ref{CM}, there exists a good filtration $\mathscr{F}$ on $R=\bigoplus_{m\ge0} R_m$ where $R_m=H^0(X_0,mr_0L_0)$. By Theorem \ref{CM}, we see that $\lim_{k\to\infty}\frac{w_{\mathscr{F}}(k)}{k^{n+1}}=\frac{r_0^{n+1}}{(n+1)!}(L'^{n+1}-L^{n+1})$. Thus, we have that 
 \begin{align*}
 \mathrm{Ding}_{(\Delta',H)}(X',L')&-\mathrm{Ding}_{(\Delta,H)}(X,L)=\mathrm{lct}(X,\Delta+\mathfrak{a};X_0)-1-\lim_{k\to\infty}\frac{w_{\mathscr{F}}(k)}{r_0k\,\mathrm{dim}\,R_k}.
 \end{align*} 
For the definition of 
\[
\mathrm{lct}(X,\Delta+\mathfrak{a};X_0)=\sup\{t\in\mathbb{Q}|(X,\Delta+\mathfrak{a}+tX_0)\textrm{ is sublc around }X_0\},
\] we refer to \cite[Definition 2.6]{Fjt2}. Since $\mathcal{O}_{C,0}$ is a discrete valuation ring, there exist free bases $\{s_1,\cdots,s_{N_k}\}$ of $\pi_*\mathcal{O}(kr_0L)\otimes_{\mathcal{O}_C}\mathcal{O}_{C,0}$ where $N_k=\mathrm{dim}\,R_k$ and $\{s'_1,\cdots,s'_{N_k}\}$ of $\pi'_*\mathcal{O}(kr_0L')\otimes_{\mathcal{O}_C}\mathcal{O}_{C,0}$ such that $s'_i=t^{\lambda_i}s_i$ for some $\lambda_i\in\mathbb{Z}_{\ge0}$. Here, $t$ is a generator of the maximal ideal of $\mathcal{O}_{C,0}$ and recall that any homomorphism of free $\mathcal{O}_{C,0}$-modules is represented by a diagonal matrix. By Theorem \ref{CM}, $w_{\mathscr{F}}(k)=-\sum_{i=1}^{N_k}\lambda_i+O(k^{n-1})$. For sufficiently large $k$, we may assume that $\delta_{r_0k}(X_0,\Delta_0,L_0)>1$ by \cite[Theorem A]{BlJ}. Let $\mathrm{div}(s_i)=D_i$. Then, consider $s'_i\in\pi_*\mathcal{O}(kr_0L)\otimes_{\mathcal{O}_C}\mathcal{O}_{C,0}$ and we have
 \[
 \mathrm{lct}(X,\Delta+\mathfrak{a};X_0)\ge \mathrm{lct}\left(X,\Delta+\frac{1}{r_0kN_k}\sum_{i=1}^{N_k}\mathrm{div}(s'_i);X_0\right).
 \]
 Here, $(X_0,\Delta_0+\frac{1}{r_0kN_k}\sum_{i=1}^{N_k}D_{i,0})$ is klt since $\frac{1}{r_0kN_k}\sum_{i=1}^{N_k}D_{i,0}$ is an $r_0k$-basis type divisor with respect to $L_0$. By the inversion of adjunction \cite[Theorem 5.50]{KM}, $(X,\Delta+\frac{1}{r_0kN_k}\sum_{i=1}^{N_k}D_i+X_0)$ is lc around $X_0$ for sufficiently divisible $k$. On the other hand,
 \[
 \frac{1}{r_0kN_k}\sum_{i=1}^{N_k}\mathrm{div}(s'_i)+\left(1-\frac{1}{r_0kN_k}\sum_{i=1}^{N_k}\lambda_i\right)X_0=\frac{1}{r_0kN_k}\sum_{i=1}^{N_k}D_i+X_0.
 \]
 Hence we obtain $$\mathrm{lct}(X,\Delta+\mathfrak{a};X_0)-1+\sum_{i=1}^{N_k}\frac{\lambda_i}{r_0kN_k}\ge0$$ for such $k$. Therefore, it follows that $ \mathrm{Ding}_{(\Delta',H)}(X',L')\ge \mathrm{Ding}_{(\Delta,H)}(X,L)$.
 
 Finally, we treat the case when $\delta(X_0,\Delta_0,L_0)=1$. We have shown that   
 \[
 \mathrm{Ding}_{(\Delta',H+\epsilon L)}(X',(1-\epsilon)L')\ge \mathrm{Ding}_{(\Delta,H+\epsilon L)}(X,(1-\epsilon)L)
 \]
  for sufficiently small $\epsilon>0$ since $\delta(X_0,\Delta_0,(1-\epsilon )L_0)=(1-\epsilon)^{-1}>1$. It is easy to see that $D_{(X',\Delta',H+\epsilon L,(1-\epsilon)L')}$ depends on $\epsilon$ continuously and so does $\mathrm{lct}(X',\Delta'+D_{(X',\Delta',H+\epsilon L,(1-\epsilon)L')};X'_0)$. Thus, we have 
  \begin{align*}
 &\mathrm{Ding}_{(\Delta',H)}(X',L')- \mathrm{Ding}_{(\Delta,H)}(X,L)\\
&=  \lim_{\epsilon\to0} (\mathrm{Ding}_{(\Delta',H+\epsilon L)}(X',(1-\epsilon)L')- \mathrm{Ding}_{(\Delta,H+\epsilon L)}(X,(1-\epsilon)L))\ge0.
  \end{align*}
  We complete the proof.
 \end{proof}

\subsubsection{Proof of Theorem \ref{mt2}}\label{3.3.3}

We combine these results to prove Conjecture \ref{CMconj} when $(X_0,\Delta_0,L_0)$ is specially K-(semi)stable. The following is Theorem \ref{mt2}.
 
 \begin{thm}[CM minimization for special K-stability]\label{CM2}
 Let $\pi:(X,\Delta,L)\to C$ and $\pi':(X',\Delta',L')\to C$ be two polarized log families over a proper smooth curve such that there exists a $C^{\circ}$-isomorphism $f^{\circ}:(X,\Delta,L)\times_CC^{\circ}\cong (X',\Delta',L')\times_CC^{\circ}$. Suppose that $K_X+\Delta$ is $\mathbb{Q}$-Cartier and $(X_0,\Delta_0,L_0)$ is specially K-semistable.
 
 Then  \[
 \mathrm{CM}((X',\Delta',L')/C)\ge\mathrm{CM}((X,\Delta,L)/C).
 \]
 Furthermore, if $(X_0,\Delta_0,L_0)$ is specially K-stable, then equality holds iff $f^{\circ}$ can be extended to a $C$-isomorphism $f:(X,\Delta,L)\cong (X',\Delta',L')$.
 \end{thm}
 
 \begin{proof}
{\it Case 1. $(X_0,\Delta_0)$ is klt}. 
Suppose first that $(X_0,\Delta_0,L_0)$ is specially K-stable. As in the proof of Theorem \ref{CM}, we may assume that there exists a birational morphism $\mu:X'\to X$ such that $\mu|_{X'^\circ}=f^{\circ-1}$. Replacing $L$ by $(\delta(X_0,\Delta_0,L_0)-\epsilon)L$ for sufficiently small $\epsilon>0$ such that $\delta(X_0,\Delta_0,L_0)-\epsilon\in\mathbb{Q}$, we may also assume that $\delta(X_0,\Delta_0,L_0)\ge1$, $(X_0,\Delta_0,L_0)$ is uniformly $\mathrm{J}^{K_{X_0}+\Delta_0+ L_0}$-stable and $K_{X_0}+\Delta_0+ L_0$ is ample. Let $H=-(K_X+\Delta+L)$. Then, 
 \begin{align*}
  \mathrm{CM}((X',\Delta',L')/C)&= \mathrm{CM}((X',\Delta',L')/C)+ \mathcal{J}^H(X',L')- \mathcal{J}^H(X',L')\\
  &=\mathrm{CM}_{(\Delta',H)}(X',L')- \mathcal{J}^H(X',L')\\
  &\ge\mathrm{Ding}_{(\Delta',H)}(X',L')+ \mathcal{J}^{K_X+\Delta+L}(X',L'),\\
  \mathrm{CM}((X,\Delta,L)/C)&=\mathrm{CM}_{(\Delta,H)}(X,L)- \mathcal{J}^H(X,L)\\
  &=\mathrm{Ding}_{(\Delta,H)}(X,L)+ \mathcal{J}^{K_X+\Delta+L}(X,L).
 \end{align*}
Here, we applied Lemma \ref{ding.lemma} (3). By Theorem \ref{ding}, $\mathrm{Ding}_{(\Delta',H)}(X',L')\ge\mathrm{Ding}_{(\Delta,H)}(X,L)$. On the other hand, $\mathcal{J}^{K_X+\Delta+L}(X',L')\ge\mathcal{J}^{K_X+\Delta+L}(X,L)$ and hence we have the desired inequality by Proposition \ref{jcm2}. Furthermore, equality holds iff $\mu^*L\sim_{\mathbb{Q},C}L'$.

When $(X_0,\Delta_0,L_0)$ is specially K-semistable, it is easy to see that $$\mathrm{CM}_{(\Delta',\epsilon L)}(X',L')\ge \mathrm{CM}_{(\Delta,\epsilon L)}(X,L)$$ for $\epsilon>0$ by the same argument as above. By taking the limit $\epsilon\to 0$, we obtain the desired inequality.

 {\it Case 2. General case}. We may assume that $(X_0,\Delta_0,L_0)$ is not klt but slc. Now, we may assume that $\pi'$ is a semiample family such that there exists a birational morphism $\mu:X' \to X$ such that $\mu|_{X'^\circ}=f^{\circ-1}$. Applying Proposition \ref{jcm2} to the case when $H=K_{X/C}+\Delta$, it suffices to show that $$(K_{X'}+\Delta'-\mu^*(K_X+\Delta))\cdot L'^n\ge 0.$$ Indeed, $L'$ is semiample over $X$ and $K_{X'}+\Delta'-\mu^*(K_X+\Delta)$ is effective since $(X,\Delta+X_0)$ is lc by \cite{Ka}.
 \end{proof}
 
 Finally, we obtain the following corollaries. First, we apply Theorem \ref{CM} to test configurations and obtain:
 \begin{cor}\label{corf1}
 Let $(X,\Delta,L)$ be a klt polarized pair. If $(X,\Delta,L)$ is specially K-stable (resp., semistable), then $(X,\Delta,L)$ is uniformly K-stable (resp., K-semistable).
 \end{cor}
 
 We prove $\clubsuit$ for specially K-stable pairs in \S\ref{intro}.
 
 \begin{cor}[Separatedness of $\mathbb{Q}$-Gorenstein specially K-stable families]\label{sep}
 Let $\pi:(X,\Delta,L)\to C$ and $\pi':(X',\Delta',L')\to C$ be two polarized $\mathbb{Q}$-Gorenstein families over a smooth affine curve such that there exists a $C^{\circ}$-isomorphism $f^{\circ}:(X,\Delta,L)\times_CC^{\circ}\cong (X',\Delta',L')\times_CC^{\circ}$. Suppose that $(X_0,\Delta_0,L_0)$ is specially K-stable and $(X'_0,\Delta'_0,L'_0)$ is specially K-semistable.
 
 Then  $f^{\circ}$ can be extended to a $C$-isomorphism $f:(X,\Delta,L)\cong (X',\Delta',L')$.
 \end{cor}
 
 \begin{proof}
 Note that $C$ is not proper in general. However, by properness of Hilbert schemes, we compactify the family $\pi:(X,\Delta,L)\to C$ to the one over a smooth proper curve $\overline{C}$. If necessary, take a suitable blow up of the compactification of $X$ and we may assume that this is $\mathbb{Q}$-Gorenstein. On the other hand, if $0'\in\overline{C}\setminus C$, we compactify $\pi'$ in the same way as $\pi$ around $0'$. Thus, we may assume that $C$ is proper. Then we apply Theorem \ref{CM2} and obtain the assertion.
 \end{proof}
 
 \begin{cor}\label{aut}
 Let $(X,\Delta,L)$ be a specially K-stable lc pair. Then $\mathrm{Aut}\,(X,\Delta,L)$ is a finite group. Moreover, the identity component $\mathrm{Aut}_0\,(X,\Delta)$ of $\mathrm{Aut}\,(X,\Delta)$ is an Abelian variety.
 \end{cor}
 \begin{proof}
 The first assertion follows from Corollary \ref{sep} and the same argument of \cite[Corollary 3.5]{BX}. The rest also follows as \cite[Corollary 1.6]{O} from the fact that the action of the linear algebraic group $\mathrm{Aut}_0\,(X,\Delta)$ on $\mathrm{Pic}^0(X)$ defines a quasi-finite morphism $\mathrm{Aut}_0\,(X,\Delta)\to\mathrm{Pic}^0(X)$ by the first assertion.
 \end{proof}

\end{document}